\documentclass{article}

\usepackage[english]{babel}
\usepackage[utf8]{inputenc}
\usepackage{bbm}
\usepackage{enumitem}
\usepackage{amsmath}
\usepackage{amssymb}
\usepackage{comment}
\usepackage{algorithm}
\usepackage{algorithmic}
\usepackage{amsthm}
\usepackage{todonotes}
\usepackage{placeins}
\usepackage{vmargin}
\usepackage[compact]{titlesec}
\usepackage{graphicx}
\usepackage{xcolor}
\usepackage{transparent}
\usepackage{leftidx}
\usepackage{url}

\usepackage{pgfplots}
\pgfplotsset{compat=1.18}
\usepgfplotslibrary{colormaps}
\usepgfplotslibrary{patchplots}
\usepgfplotslibrary{fillbetween}
\usepgfplotslibrary{external}
\usepgfplotslibrary{groupplots}

\titleformat{\section}[hang]
{\filcenter\bfseries}
{\thesection. }{0pt}{}
\titleformat{\subsection}
{\bfseries}
{\thesubsection. }{0pt}{}
\newcommand{\B}{\mathcal{B}}
\newcommand{\E}{\mathcal{E}}

\newcommand{\M}{\mathcal{Y}}
\newcommand{\X}{\mathcal{X}}

\newcommand{\Y}{\mathcal{Y}}
\newcommand{\R}{\mathbb{R}}
\renewcommand{\L}{\mathcal{L}}
\DeclareMathOperator{\Vertical}{\kappa}

\newtheorem{theorem}{Theorem}
\newtheorem{proposition}{Proposition}[section]
\newtheorem{bemerkung}[theorem]{Remark}
\newtheorem{example}[theorem]{Example}
\newtheorem{definition}[proposition]{Definition}
\newtheorem{lemma}[proposition]{Lemma}
\newtheorem{corollary}[proposition]{Corollary}

\makeatletter
\newcommand{\oset}[3][0ex]{%
  {\mathop{#3}\limits^{
    \vbox to#1{\kern-2\ex@
    \hbox{\!$#2$}}}}}
\makeatother

\newcommand{\NE}[1]{{\big\langle{#1}\big\rangle}}

\newcommand{\VT}[2]{\oset[-0.35ex]{\scriptscriptstyle#2}{V}_{\!\!#1}}
\newcommand{\VTprime}[2]{\oset[-0.6ex]{\scriptscriptstyle#2}{V'}_{\!\!\!\!#1}}

\setlist{leftmargin=2em}

\title{Newton's method for nonlinear mappings into vector bundles}

\author{Laura Weigl \& Anton Schiela}

\begin{document}
\maketitle
\begin{abstract}
We consider Newton's method for finding zeros of mappings from a manifold $\X$ into a vector bundle $\E$. In this setting a connection on $\E$ is required to render the Newton equation well defined, and a retraction on $\X$ is needed to compute a Newton update. We discuss local convergence in terms of suitable differentiability concepts, using a Banach space variant of a Riemannian distance. We also carry over an affine covariant damping strategy to our setting. Finally, we will illustrate our results by applying them to generalized non-symmetric eigenvalue problems and providing a numerical example.

\end{abstract}

\section{Introduction}
Newton's method is a central algorithm for the solution of nonlinear problems, but also, in its variants, a theoretical tool, e.g., in the proof of the implicit function theorem. Most of the literature consider Newton's method for finding zeros $F(x)=0$ for a differentiable mapping $F:X \to Y$, where $X$ and $Y$ are linear, normed, not necessarily finite dimensional spaces, in most cases Banach spaces. In this setting, a Newton step is defined as the solution $\delta x\in X$ of a linear operator equation together with an additive update:
\[
  F'(x)\delta x +F(x) =0,  \quad x_+ = x+\delta x.
\]
 Extensions of Newton's method to problems $F : \X \to Y$, where $\X$ is a Riemannian manifold and $Y$ still a linear space with the help of so called retractions $R_x:T_x\X \to \X$, are relatively straightforward, replacing the additive update by $x_+=R_x(\delta x)$. But also, Newton's method has been used to find stationary points of twice differentiable functions $f : \X \to \R$ \cite{absil2008optimization}, and to find zeros of vector fields $v\in \Gamma(\X)$ on Riemannian manifolds $\X$ (cf. e.g. \cite{gabay1982minimizing, ArgyrosHilout:2009:1,FernandesFerreiraYuan:2017:1,louzeiro2025inexact}) or one-forms \cite{smith1993geometric,smith1994optimization}. For an account on Newton's method for vector fields and further literature we refer to \cite[Section 6]{absil2008optimization}. In the semismooth context, Newton's method for vector fields on Riemannian manifolds has already been discussed in \cite{diepeveen2021inexact,DeOliveiraFerreira:2020:1,si2024riemannian}. Newton's method for shape optimization was considered in \cite{Schulz2015}. 
 
Rich classes of problems are formulated in a manifold setting, such as constrained optimization \cite{bergmann2022first,liu2020simple, schiela2021sqp, Schulz2015}, non-symmetric variational problems, problems of stationary action \cite{MarsdenRaitu:1994}, or other saddle point problems, to name just a few.
Given this variety of settings, the following questions arise: What is an appropriate general mathe\-matical framework for Newton's method for mappings between nonlinear spaces? What structure is actually needed to define, but also to implement and globalize Newton's method for such problems? It is the aim of this work to explore possible answers to these questions. 
 
Specifically, we want to establish Newton's method for a mapping $F: \X \to \E$, where $\X$ is a differentiable manifold and $p : \E \to \M$ is a \emph{vector bundle} with fibres $E_y$, $y\in \M$. This setting still allows to formulate root-finding problems, which now read $F(x)=0_{p(F(x))}\in E_{p(F(x))}$, but it is general enough to cover the above mentioned problem classes. 

In contrast to the classical case the linear space $E_{p(F(x))}$ in which the residuals are evaluated is not constant, but depends on $x$ via $p\circ F$. 
Moreover, the derivative of $F: \X\to \E$ is now a mapping $F'(x) : T_x\X \to T_{F(x)}\E$, which means that the Newton equation cannot be formulated in the usual way, since the codomains of $F$ and $F'(x)$, i.e., $\E$ and $T_{F(x)}\E$, do not coincide. It is thus necessary to introduce some additional geometric structure, namely a \emph{connection} on $\E$. A conncetion induces for every $e\in \E$ a linear mapping $Q_e : T_e\E \to E_{p(e)}$, which plays the role of a projection onto the corresponding fibre. This allows us to formulate the Newton equation in a well defined way: 
\[
  Q_{F(x)} \circ F'(x)\delta x +F(x) =0_{p(F(x))}.
\]
If $Q_{F(x)} \circ F'(x)$ is invertible, the Newton direction $\delta x \in T_x\X$ can be computed, and then the Newton step can be defined via a retraction $x_+=R_x(\delta x)$.

To observe convergence of Newton's method we need a \emph{metric} $d : \X \times \X \to \R_+$ on $\X$. Since Newton's method  in the linear setting is not restricted to Hilbert spaces and has important applications in Banach spaces (cf. e.g. \cite{ulbrich2011semismooth} or \cite{Hinze:2022}), we will introduce metrics similar to, but more general than Riemannian distances.
 Finally, to compare images of successive iterates on different fibres,  e.g.,  $F(x_+)\in E_{p(F(x_+))}$ and $F(x)\in E_{p(F(x))}$, \emph{vector back-transports} of the form $\VT{y}{\leftarrow} : \E \to E_y$ are needed on $\E$ (we call them back-transports, because they work in the opposite direction than classical vector transports).

After setting up this framework, we will establish basic local convergence results, based on minimal assumptions on the smoothness of $F$, using a geometric version of Newton differentiability. Besides giving a classical a-priori result, we also establish a theorem, which relates local convergence to a quantity that is purely defined in terms of the space of iterates $\X$, and which can be estimated a-posteriori. This opens the door to define an affine covariant damping strategy in the spirit of Deuflhard~\cite{Deuflhard}, which is based on following the so called Newton path. While carrying over this algorithm from linear spaces to our nonlinear setting, we will observe that strict differentiability of $F$ is required together with a consistency condition of the employed connection and the vector back-transport to make this strategy viable. Finally, we will present an application together with a numerical example concerning the solution of non-symmetric generalized eigenvalue problems. 
More extensive numerical results are presented in a companion paper, cf. \cite{weigl2025newton}. There, we discuss the application of Newton's method for solving variational problems on manifolds and illustrate our results by the numerical computations of so called \emph{elastic geodesics} and the deformation of an inextensible rod.

\section{Preliminaries}
We consider a \textit{Banach manifold} $\X$, i.e. a topological space $\X$ and a collection of charts $(U,\phi)$, where each chart $\phi : U \to \mathbb X$  maps an open subset $U$ of $\X$ homeomorphically into a real Banach space $\mathbb{X}$ with norm $\Vert\cdot\Vert_{\mathbb{X}}$.
Unless otherwise noted we will assume that the manifold is of class $C^1$, which means that the transition mappings, i.e., $t_{ji}:=\phi_j\circ \phi_i^{-1} : \phi_i(U_i\cap U_j)\to \mathbb X$ for charts $\phi_i,\phi_j$, are local diffeomorphisms and thus, in particular, locally Lipschitz. Mappings $F : \X \to \Y$ between two Banach manifolds are called continuous, differentiable, locally Lipschitz, if their representations in charts, i.e. the composition with charts $\phi_{\Y} \circ F \circ \phi_{\X}^{-1}$, have the respective property. It can readily be shown that these properties are independent of the choice of charts. However, we cannot assign a local Lipschitz constant to a locally Lipschitz mapping, since this quantity is chart dependent. 

For a manifold $\X$ we define tangent spaces by using charts, see, e.g., \cite[II §2]{Lang}. Let $x\in\X$ and consider the set 
\[ \lbrace (\phi,v) \; \vert \; \phi:U\rightarrow \mathbb{X} \text{ chart at } x\in\X, \ v\in\mathbb{X} \rbrace. \]
The relation $(\phi_i,v_i)\sim_x(\phi_j,v_j) \ :\Leftrightarrow \ t_{ji}^\prime(\phi_i(x))v_i = v_j$ defines an equivalence relation. The corresponding equivalence class is called a tangent vector $v$ of $\X$ at $x$. The set of these equivalence classes at a point $x\in\X$ is a vector space, called the tangent space $T_x\X$ of $\X$ at $x$. 
 Every differentiable map $F:\X\rightarrow\M$ between two $C^1$-manifolds $\X$ and $\M$ induces a linear map $F^\prime(x) : T_{x}\X \rightarrow T_{F(x)}\M$, termed \textit{tangent map} or \textit{derivative} of $F$ at any $x\in\X$ \cite{Lang}.
 \paragraph{Vector bundles.}
 
Consider a \emph{vector bundle} $p:\E \rightarrow \M$, which will be assumed to be of class $C^1$, unless otherwise noted. Its \textit{projection} $p$ is a surjective $C^1$-map, which assigns to $e\in \E$ its base point $y=p(e)$ in the \textit{base manifold} $\M$ \cite[III §1]{Lang}. The \emph{total space} $\E$ is a manifold with special structure: for each $y\in \M$ the preimage $E_y:=p^{-1}(y)$ is a Banachable space (i.e. a complete topological vector space whose topology is induced by some norm, rendering $E_y$ complete), called \textit{fibre} over $y$. The zero element in $E_y$ is denoted by $0_y$. 
A mapping $v : \M \to \E$, such that $p(v(y))=y$ for all $y\in \M$ is called \textit{section} of $\E$. The set of all sections is denoted by $\Gamma(\E)$. The most prominent example for a vector bundle is the tangent bundle $\pi: T\X\to \X$ of a manifold $\X$ with fibres $T_x\X$. A section of $T\X$ is called a \textit{vector field}.

Vector bundles can be described via general local charts, but usually, \emph{local trivializations} are used instead, which reflect their structure better. For a Banach space $(\mathbb E,\|\cdot\|_{\mathbb E})$ and an open set $U \subset \M$, a local trivialization is a diffeomorphism $\tau : p^{-1}(U) \to U \times \mathbb E$, such that $\tau_y:=\tau|_{E_y} \in L(E_y,\mathbb E)$ is a linear isomorphism. Thus, any element of $e\in \E$ can be represented as a pair $(y,e_\tau)\in \M\times \mathbb E$ with $y=p(e)$ and $e_\tau$ depending on the trivialization. For two trivializations $\tau $ and $\tilde \tau$ around $y\in\M$ we can define smooth transition mappings $A : y \mapsto \tilde \tau_y\tau_y^{-1}$ with $A(y) \in L(\mathbb E,\mathbb E)$, and obtain $e_{\tilde \tau} = A(y)e_\tau$.
If $\phi : U \to \mathbb Y$ is a chart of $\M$, then we can construct a \emph{vector bundle chart}: 
\begin{align*}
   \Phi_{\E} : p^{-1}(U) &\to \mathbb Y \times \mathbb E, \quad 
   e \mapsto (\phi(p(e)),\tau_{p(e)}e),
\end{align*}
which retains the linear structure of the fibres. Vice versa, the second component of a vector bundle chart yields a local trivialization. 

\begin{bemerkung}
    \label{rem:VBCharts}
    For the tangent bundle $\pi: T\X \to \X$ of a manifold $\X$ we have trivializations and
    vector bundle charts in the following natural way:\\
    Let $(U_i,\phi_i)$ be a chart of $\X$ at $x$. Then every tangent vector $v\in T_x\X$ has a chart representation $v_i \in \mathbb{X}$, so $\tau_i(v):=v_i$ is a trivialization. Since tangent vectors transform by $t_{ji}^\prime(\phi_i(x))v_{i} = v_{j}$ under change of charts, the transition mapping is given by $A(x)=t_{ji}^\prime(\phi_i(x))$.
 A vector bundle chart for $T\X$ is then given by:
    \begin{align*}
        \Phi_{T\X} : \pi^{-1}(U_i) &\to \mathbb{X} \times \mathbb{X}, \quad 
        (x,v) \mapsto (\phi_i(x), v_i).
    \end{align*}
\end{bemerkung}

Similarly, each element $\delta e \in T_e\E$ can be represented in local trivializations as a pair of elements $(\delta y,\delta e_\tau) \in T_y\M \times \mathbb E$, where the tangential part $\delta y = p'(e)\delta e$ is given canonically. The representation $\delta e_\tau$ of the fibre part, however, depends crucially on the chosen trivialization, and so there is no \emph{natural} splitting of $\delta e$ into a tangential part and a fibre part. In fact, if $e_{\tilde\tau} = A(y)e_\tau$, then $\delta e_{\tilde\tau} = A(y)\delta e_\tau+(A'(y)\delta y) e_\tau$ by the product rule. 

Given a mapping $f:\X\to \M$ and a vector bundle $p: \E\to \Y$ we denote by $f^*p : f^*\E \to \X$ the \emph{pullback bundle} of $\E$ via $f$, i.e. $(f^*E)_x:= E_{f(x)}$, or $f\circ (f^*p)=p$. 

\paragraph{Fibrewise linear mappings.}
Consider two vector bundles $p_1 : \E_1 \to \M_1, \; p_2 : \E_2 \to \M_2$, and denote by $L(E_{1,y},E_{2,z})$ the set of continuous linear mappings $E_{1,y}\to E_{2,z}$ for $y\in \M_1$ and $z\in\M_2$. 
Given a mapping $s:\M_1\to \M_2$ we can collect all linear mappings in $L(E_{1,y},E_{2,s(y)})$ for all $y\in \M_1$ in a set $\L(\E_1,s^*\E_2)$. We observe that $p_\L : \L(\E_1,s^*\E_2) \to \M_1$ is a vector bundle itself with fibres $L(E_{1,y},E_{2,s(y)})$. In trivializations an element $H$ of $\L(\E_1,s^*\E_2)$ is represented by a pair $(y, H_\tau)$ where $y\in\M_1$ and $H_\tau \in L(\mathbb E_1,\mathbb E_2)$. If the co-domain is a fixed fibre, i.e, $E_z \cong \{z\}\times E_z$, we use the simplified notation $\L(\E_1,E_z)$ instead of the cumbersome $\L(\E_1,s_z^*(\{z\}\times E_z))$, where $s_z:\M_1\to \{z\}$ is the constant map. 

We define a \emph{fibrewise linear mapping}, also known as \emph{vector bundle morphism}, see, e.g., \cite[III §1]{Lang}, as a section $S \in \Gamma(\L(\E_1,s^*\E_2))$, i.e, a mapping 
\[
S: \M_1 \to \L(\E_1,s^*\E_2) \text{ with } p_\L \circ S=Id_{\M_1}. 
\]
$S$ is called locally bounded, respectively differentiable, if its representation in local trivializations
\begin{equation}\label{eq:FibrewiseTriv}
S_\tau : U \to L(\mathbb E_1,\mathbb E_2), \quad S_\tau(y) = \tau_{\L(\E_1, s^*\E_2)}(S):= \tau_{\E_2,s(y)}\circ S(y) \circ \tau_{\E_1,y}^{-1} \in L(\mathbb E_1,\mathbb E_2)
\end{equation}
has the respective property. 

If $S$ is differentiable, then its derivative is a section $S'\in \Gamma(T\L(\E_1,s^*\E_2))$ of the tangent bundle $p_\L' : T\L(\E_1,s^*\E_2)\to T\M_1$, i.e., a mapping 
\[
S' : T\M_1 \to T\L(\E_1,s^*\E_2), \text{ such that } p_\L' \circ S'=Id_{T\M_1}.
\]
From $S \in \Gamma(\L(\E_1,s^*\E_2))$ we can construct a mapping $\NE{S} \in C^1(\E_1,\E_2)$ as follows:
\begin{align*}
    \NE{S} : \E_1 &\to \E_2, \quad 
                    e \mapsto \NE{S}(e):=S(p_1(e))e.
\end{align*}
Thus, we obtain a \textit{natural inclusion} of the sections of the bundle of fibrewise linear mappings into the $C^1$-mappings: 
\[
\varphi : \Gamma(\L(\E_1, s^*\E_2)) \to C^1(\E_1, \E_2), \quad S \mapsto \NE{S}.
\]
Comparison of the derivatives $S'\in \Gamma(T\L(\E_1,s^*\E_2))$ and $\NE{S}' \in \Gamma(\L(T\E_1,S^*T\E_2))$ in trivializations yields by the product rule \begin{equation}
    \label{trivializationDerivative}
    \NE{S}_\tau'(y,e_\tau)(\delta y, \delta e_\tau) = S_\tau(y)\delta e_\tau+(S_\tau'(y)\delta y) e_\tau  \quad \forall (\delta y,\delta e_\tau) \in T_{y}\M_1\times \mathbb E_1.
\end{equation}
For example, if $F: \X \to \M$ is twice differentiable, then the derivative $S:=F'\in \Gamma(\L(T\X,F^*T\M))$ is a differentiable fibrewise linear mapping with $\langle F'\rangle \in C^1(T\X,T\M)$. 

\paragraph{Vector transports.} Special fibrewise linear mappings, which will later be needed in our globalization strategy, are the so called \emph{vector transports}.
\begin{definition}
For a vector bundle $p:\E\rightarrow\M$ and $y \in \M$ we define a \emph{vector transport} as a section $\VT{y}{\rightarrow} \in \Gamma(\L(\M \times E_{y}, \E))$, i.e.
\[
    \VT{y}{\rightarrow} : \M \to \L(\M \times E_{y}, \E), \ \text{with } p_\L(\VT{y}{\rightarrow}(\hat{y})) = \hat{y} \; \forall \hat{y}\in\M,
\]
with the properties that $\VT{y}{\rightarrow}(y)=Id_{E_{y}}$ and $\VT{y}{\rightarrow}(\hat y)$ is invertible for all $\hat y\in \M$. \\
We define a \emph{vector back-transport} as a section $\VT{y}{\leftarrow}\in \Gamma(\L(\E, E_{y}))$, i.e.
\[
    \VT{y}{\leftarrow} : \M \to \L(\E, E_y), \ \text{with } p_\L(\VT{y}{\leftarrow}(\hat{y})) = \hat{y} \; \forall \hat{y}\in\M
\]
with the property that $\VT{y}{\leftarrow}(y)=Id_{E_{y}}$ and $\VT{y}{\leftarrow}(\hat y)$ is invertible for all $\hat y\in \M$. \\
\end{definition}
By restricting $\M$ appropriately, vector transports can also be defined locally. 
Given a vector transport $\VT{y}{\rightarrow} \in \Gamma(\L(\M\times E_y, \E))$, we can derive a vector back-transport $\VT{y}{\leftarrow}\in \Gamma(\L(\E, E_{y}))$ by taking fibrewise inverses $\VT{y}{\leftarrow}(\hat{y}) := (\VT{y}{\rightarrow}(\hat{y}))^{-1}$, $\hat{y}\in\M$, and vice versa. Often, vector transports are given as mappings $V(\cdot,\cdot) : \M\times \M \to \L(\E,\E)$ with $V(y,\hat y)\in L(E_y,E_{\hat y})$. These admit interpretations as vector transports or as vector back-transports as follows: $\VT{y}{\rightarrow}(\hat y)=V(y,\hat y)=\VT{\hat{y}}{\leftarrow}(y)$. Nevertheless, the derivatives of these two objects show subtle, but important differences.

\paragraph{Connections on vector bundles.}
The concept of a \emph{connection} on a vector bundle $p: \E\to \M$ is fundamental to differential geometry. It is an additional geometric structure imposed on $\E$ that describes in an infinitesimal way, how neighboring fibres are related.  The most prominent example is the \textit{Levi-Civita connection} on $\E=T\X$ (see, e.g., \cite[VIII §4]{Lang}) for a Riemannian manifold $\X$.
Connections give rise to further concepts like the covariant derivative, curvature, geodesics, and parallel transports. They are described in the literature in various equivalent ways (cf., e.g., \cite[IV §3 or X §4 or XIII]{Lang}). For the purpose of our paper we choose a formulation that emphasizes the idea of a connection map $Q_e : T_e\E \to E_{p(e)}$, which then induces a splitting of $T_e\E$ into a vertical and a horizontal subspace.

It is well known that the kernel of $p^\prime \in \Gamma(\L(T\E, p^*T\M))$ canonically defines the \emph{vertical subbundle} $V\E$ of $\E$ \cite[IV §3]{Lang}. Its fibres are closed linear subspaces, called the \textit{vertical subspaces} $\mathrm{Vert}_e := \ker p^\prime(e) \subset T_e\E$.
We can identify $E_{p(e)} \cong  \mathrm{Vert}_e$ canonically by the isomorphism 
$\Vertical_e : E_{p(e)} \rightarrow \mathrm{Vert}_e$, given by $w \mapsto \frac{d}{dt}(e+tw)\vert_{t=0}$. Elements of $\mathrm{Vert}_e$ are represented in trivializations by pairs of the form $(0_{y},\delta e_\tau)\in T_y\M \times \mathbb E$.

We define a \emph{connection map} as a section
$Q\in \Gamma(\L(T\E,p^*\E))$, i.e., for each $e\in\E$ we get a continuous linear map $Q_e \in L(T_e\E,E_{p(e)})$,
with the additional property that $Q_e \Vertical_e = Id_{E_{p(e)}}$. The kernels $\mathrm{Hor}_e := \ker Q_e$ are called \emph{horizontal subspaces}, collected in the \emph{horizontal subbundle} $H\E$. While $V\E$ is given canonically, a choice of $Q$ (or equivalently $H\E$), imposes additional geometric structure on $\E$. Using the identification $\mathrm{Vert}_e \cong E_{p(e)}$, we can view $Q$ as a fibrewise ``projection'' onto $\E$, with $\langle Q \rangle : T\E \rightarrow \E$. 
In trivializations $(y, e_\tau) \in \M\times \mathbb E$ for $e\in\E$ and $(\delta y, \delta e_\tau) \in T_y\Y \times \mathbb E$ for $\delta e\in T_e\E$, the representation for a connection map $Q$ can be written in the form
\begin{equation}\label{eq:Qtriv}
    Q_e\delta e \sim Q_{e,\tau}(\delta y,\delta e_\tau)=\delta e_\tau - B_{y,\tau}(e_\tau)\delta y, 
\end{equation}
where $B_{y,\tau} : \mathbb E \to L(T_y\M,\mathbb E)$ assigns a linear mapping $B_{y,\tau}(e_\tau)$ to each $e_\tau \in \mathbb E$. Since $Q_{e,\tau}(0_y,\delta e_\tau)=\delta e_\tau$, we see that $Q_{e,\tau}$ indeed represents a projection onto $\mathrm{Vert}_e$. 
To reflect the linearity of the fibres we want the connection to be \emph{linear} (fibrewise with respect to $e$). To this end, we consider the fibrewise scaling $m_s : \E\rightarrow\E, e \mapsto se$ and require the condition $Q \circ m_s^\prime = m_s \circ Q$ for all $s\in\mathbb{R}$,
which reads in trivializations:
\[
   Q_{se,\tau}(\delta y,s\delta e_\tau)=sQ_{e,\tau}(\delta y,\delta e_\tau) \quad \forall s\in \mathbb R. 
\]
It can be shown that $Q$ is linear, if and only if the mapping $e_\tau \mapsto B_{y,\tau}(e_\tau)$ is linear, or put differently, the mapping $(\delta y,e_\tau) \mapsto B_{y,\tau}(e_\tau)\delta y$ is bilinear. Alternatively, for a vector bundle chart $\Phi$ we have a representation $B_{y,\Phi}:\mathbb Y\times \mathbb E\to \mathbb E$.

If $\E=T\X$ in classical Riemannian geometry, where $Q$ is given by the Levi-Civita connection, the bilinear mapping $B_{x,\Phi}:\mathbb X \times \mathbb X\to \mathbb X$ is represented by Christoffel symbols, and it is symmetric, i.e., $B_{x,\Phi}(v_i)w_i=B_{x,\Phi}(w_i)v_i$, in natural tangent bundle charts.         
\paragraph{Connections induced by vector back-transports.}
If no Levi-Civita connection is present or expensive to evaluate algorithmically, connections can be derived from differentiable vector back-transports, which are needed frequently for numerical purposes anyway:
\begin{lemma}
\label{lem:VTconnection}
Let $\VT{y}{\leftarrow} \in  \Gamma(\L(\E, E_{y}))$ be a vector back-transport and $e\in E_y$, $y= p(e)$. Then
\[
    Q_{e}:= \NE{\VT{y}{\leftarrow}}^\prime(e):T_{e}\E \to E_{y}
\]
defines a linear connection map at $e$, which is represented in trivializations by \eqref{eq:Qtriv}  with 
\begin{equation}\label{eq:diffVtTriv}
B_{y,\tau}(e_\tau)\delta y=-(\VTprime{{y,\tau}}{\leftarrow}(y)\delta y) e_\tau.
\end{equation}
If $\VT{y}{\leftarrow}$ is defined via fibrewise inverses of a vector transport $\VT{y}{\rightarrow}\in \Gamma(\L(\M\times E_y,\E))$, then 
\begin{equation}\label{eq:diffVtTriv2}
B_{y,\tau}(e_\tau)\delta y=(\VTprime{{y,\tau}}{\rightarrow}(y)\delta y) e_\tau.
\end{equation}
\end{lemma}
\begin{proof}
By differentiation of $\NE{\VT{y}{\leftarrow}} : \E \to E_{y}$ we obtain that $\NE{\VT{y}{\leftarrow}}^\prime(e):T_e\E\to E_y$ is a continuous linear map.
Using \eqref{trivializationDerivative} and $\VT{y}{\leftarrow}(y) = Id_{E_y}$, the representation of $\NE{\VT{y}{\leftarrow}}' \in \Gamma(\L(T\E,E_y))$ in trivializations, where $\VT{y,\tau}{\leftarrow}: \M \to L(\mathbb E,\mathbb E)$, reads
\begin{equation*}
    \NE{\VT{y}{\leftarrow}}_\tau'(y,e_\tau)(\delta y, \delta e_\tau) = \delta e_\tau + \VT{y,\tau}{\leftarrow}'(y)\delta y) e_\tau \quad \forall (\delta y,\delta e_\tau),
\end{equation*}
which is of the form \eqref{eq:Qtriv}, implying \eqref{eq:diffVtTriv} and linearity. 
Let $\hat{e}\in E_y$. Since $\kappa_e(\hat{e}) \in \mathrm{Vert}_e \cong E_y$ we get the representation $\kappa_e(\hat{e})_\tau \sim (0_y, \widehat{\delta e}_\tau)$. Thus, we obtain $Q_e\kappa_e = Id_{E_y}$.

To show \eqref{eq:diffVtTriv2} we compute by the calculus rule for inverse matrices and $\VT{y,\tau}{\leftarrow}(y)=Id_{\mathbb E}$:
\begin{align*}
 \VT{y,\tau}{\leftarrow}'(y)\delta y=-\VT{y,\tau}{\leftarrow}(y)\VT{y,\tau}{\rightarrow}'(y)\delta y \, \VT{y,\tau}{\leftarrow}(y)=-\VT{y,\tau}{\rightarrow}'(y)\delta y.
\end{align*}
\end{proof}

\begin{definition}
\label{def:consistentConnection}
 Let $e\in \E$ and $y=p(e)$. We call a vector back-transport $\VT{y}{\leftarrow} \in \Gamma(\L(\E, E_y))$  \emph{consistent} with a connection map $Q \in \Gamma(\L(T\E, p^*\E))$ at $e$, if $Q_e = \NE{\VT{y}{\leftarrow}}^\prime(e)$.
\end{definition}
\begin{bemerkung}
In classical Riemannian geometry, the \emph{Levi-Civita connection} (see, e.g., \cite{Lang}) on $T\X$ is often used. This connection induces the parallel transport along geodesics and the exponential map, which may serve algorithmically as a vector transport and as a retraction. However, from an algorithmic point of view, these classical mappings, which are defined via integration, can be computationally very expensive. Then, instead of starting with a connection and deriving a vector transport from it, one can take the opposite route: start with a computationally tractable vector back-transport $\VT{y}{\leftarrow}\in\Gamma(\L(\E, E_y))$, and derive a consistent connection $Q$ by differentiation.
\end{bemerkung}
\section{Definition of Newton's method}\label{sec:DefNewton}

Let $F:\X \rightarrow \E$ be a differentiable mapping between a  Banach manifold $\X$ and a vector bundle $p:\E\to\M$. 
By the composition
\[
    y := p\circ F : \X \to \M, \; y(x) := p(F(x))\in \M,
\]
we can compute the base point $y(x) \in \M$ of $F(x)\in \E$ for given $x\in \X$. 
Consider the following root finding problem
\begin{equation*}
    F(x) = 0_{{y(x)}}.
\end{equation*}
In contrast to classical root finding problems, the linear space $E_{y(x)}$ in which $F(x)$ is evaluated now depends on $x$. Hence, when constructing iterative methods we will have to take into account that $E_{y(x)}$ changes during the iteration. 

To derive Newton's method for this problem we need to define a suitable \textit{Newton direction} $\delta x\in T_x\X$. The tangent map of $F$ is a mapping $F^\prime \in \Gamma(\L(T\X,F^*T\E))$. Since $F^\prime(x)\delta x \in T_{F(x)}\E$ and $F(x)\in E_{y(x)}$, these two quantities cannot be added, and thus the classical Newton equation $F^\prime(x)\delta x + F(x) = 0$ is not well defined. 

A \emph{connection map} $Q \in \Gamma(\L(T\E,p^*\E))$  allows us to resolve this issue.
With its help we can define the following linear mapping at $x\in \X$, which maps into the correct space:
\begin{equation}\label{eq:NewtonEquation}
    Q_{F(x)} \circ F'(x) : T_x \X \to E_{y(x)}. 
\end{equation}
Then the \emph{Newton equation} is well defined as the following linear operator equation
\begin{equation*}
Q_{F(x)} \circ F^\prime(x)\delta x + F(x) = 0_{{y(x)}}.
\end{equation*}
If $Q_{F(x)} \circ F^\prime(x)$ is invertible, then the Newton direction $\delta x \in T_x\X$ is given as the unique solution of this equation, and $\delta x=0$ holds if and only if $F(x)=0$.

The classical additive Newton update $x_+=x+\delta x$ is also not well defined since $x\in \X$ and $\delta x\in T_x\X$ cannot be added. To obtain a new iterate, we have to map $\delta x\in T_x\X$ back to the manifold $\X$, which can be done by a \textit{retraction}, a popular concept, used widely in numerical methods on manifolds (cf., e.g., \cite[Chap. 4]{absil2008optimization}).
\begin{definition}[Retraction]
    A $C^1$-\emph{retraction} on a manifold $\X$ is a $C^1$-mapping $R:T\X \rightarrow \X$, where $R_x : T_x\X\to \X$ for a fixed $x\in \X$ satisfies the following properties:
    \begin{itemize}
    \item[(i)] $R_x(0_x)=x$ 
    \item[(ii)] $R_x^\prime(0_x)=Id_{T_x\X}$
    \end{itemize}
\end{definition}
Thus, after successfully computing the Newton direction we use a local retraction to generate the \textit{Newton step} 
\begin{equation*}
    x_+ := R_x(\delta x).
\end{equation*}
The result is the following local algorithm:
\FloatBarrier
\begin{algorithm}[h]
    \caption{Newton's method on vector bundles}
        \begin{algorithmic}
            \REQUIRE $x_0$ (initial guess)
            \FOR {$k = 0,1,2,...$}
                \STATE $\delta x_k \leftarrow Q_{F(x_k)} \circ F^\prime (x_k)\delta x_k + F(x_k) = 0_{{y(x_k)}}$
                \STATE $x_{k+1} = R_{x_k}(\delta x_k)$
            \ENDFOR
    \end{algorithmic}
\end{algorithm}
\FloatBarrier
\section{Newton-Differentiability and Strict Differentiability}
Classically, quadratic convergence of Newton's method is analysed for continuously differentiable mappings, where the derivative satisfies a Lipschitz condition. However, in recent years, weaker smoothness concepts have become popular, which are sufficient to show local superlinear convergence.  
For the analysis of Newton's method we will discuss two of these differentiability concepts, which are known on linear spaces, but slightly non-standard in the context of manifolds. The first is the so called \emph{Newton-differentiability} \cite{mifflin1977semismooth,qi1993nonsmooth,ulbrich2011semismooth}, a concept which is tailored for the analysis of local convergence of Newton's method, even for semismooth problems, the second one is \emph{strict differentiability}  \cite{nijenhuis1974strong,schechter1996handbook}, a notion that is stronger than Fr\'echet differentiability, but weaker than continuous differentiability, and we will use it to study the \emph{simplified} Newton method, which is a building block of our globalization scheme, introduced below. Semismoothness of vector fields on Riemannian manifolds has already been discussed in \cite{si2024riemannian,diepeveen2021inexact,de2020newton} by using parallel transport and the exponential map. We consider mappings into general manifolds and define semismoothness and strict differentiability by looking at the problem in local charts and showing invariance with respect to changes of charts. Later, we will also view semismoothness and strict differentiability in a geometric form using general retractions and vector transports.

Consider $F:\X\to\Y$ where $\X$ and $\Y$ are both $C^1$-Banach manifolds.
Let $(U, \phi)$ be a chart of $\X$ and $(V,\psi)$ be a chart of $\Y$. 
Consider the following representation of $F$ in charts
\[
    F_{\psi,\phi} := \psi \circ F \circ \phi^{-1} : \phi(U) \to \mathbb Y.
\]
Consider also a fibrewise linear mapping $F'\in \Gamma(\L(T\X, F^*T\Y))$ with chart representations
\[
    (F_{\psi,\phi})' : \phi(U) \to L(\mathbb{X}, \mathbb{Y}).
\]
As a first step, we will state our differentiability concepts for these chart representations.
$F_{\psi,\phi}$ is called \emph{Newton-differentiable} at  $\xi_0 \in \phi(U)$ with respect to a \emph{Newton-derivative} $(F_{\psi,\phi})'$ if 
\begin{equation}
    \label{eq:NewtonDiffbarVR}
    \lim_{\xi \to \xi_0} \frac{\Vert (F_{\psi,\phi})'(\xi)(\xi-\xi_0) - (F_{\psi,\phi}(\xi)-F_{\psi,\phi}(\xi_0))\Vert_{\mathbb{Y}}}{\Vert \xi-\xi_0\Vert_{\mathbb{X}}} = 0.
\end{equation}
Newton-derivatives are, in contrast to classical derivatives, not unique, since the linearization $(F_{\psi,\phi})'(\xi)$ is evaluated at the moving point $\xi$. This is the reason, why our definition refers to a given pair $F_{\psi,\phi}, \; (F_{\psi,\phi})'$.
\begin{example}
    The real function $F(x)=|x|$ is Newton-differentiable at $x_0=0$ with respect to $F'$, if there is a neighbourhood $U$ of $x_0$, such that $F'(x)=\mathrm{sgn}\,x$ for all $x\in U\setminus \{x_0\}$.   
\end{example}
We call $F_{\psi,\phi}$ \emph{strictly differentiable} at a point $\xi_0\in\phi(U)$ if there is $(F_{\psi,\phi})'(\xi_0)\in L(\mathbb{X},\mathbb{Y})$, such that the following condition holds:\\
    For every $\varepsilon >0$ there exists a neighborhood $U_\varepsilon$ of $\xi_0$ such that
    \begin{equation}
        \label{eq:StriktDiffbarVR}
        \Vert (F_{\psi,\phi})'(\xi_0)(\xi-\eta) - (F_{\psi,\phi}(\xi) - F_{\psi,\phi}(\eta))\Vert_\mathbb{Y} < \varepsilon \Vert \xi-\eta\Vert_\mathbb{X} \quad \forall \xi,\eta \in U_\varepsilon.
    \end{equation}
    Strict differentiability is a slightly stronger notion than Fr\'echet differentiability, but weaker than continuous differentiability. 
    By the triangle inequality, we see that strictly differentiable mappings are locally Lipschitz continuous near $\xi_0$ with a Lipschitz constant that is close to $\|(F_{\psi,\phi})'(\xi_0)\|_{\mathbb X\to \mathbb Y}$. In linear spaces it is well known, and also not hard to verify, that continuous differentiability implies Newton-differentiability (w.r.t. the classical derivative), as well as strict differentiability.

    Our second step is to show independence of these definitions of the choice of charts. To this end we need the following chain rules for Newton- and strict differentiability in linear spaces:
\begin{lemma}
    Let $\mathbb{X}, \; \mathbb{Y}$ and $\mathbb{Z}$ be normed linear spaces. Consider $h:\mathbb{X} \to \mathbb{Y}$ and $g:\mathbb{Y}\to\mathbb{Z}$.
    \begin{enumerate}
        \item[(i)] Let $h$ and $g$ be Newton-differentiable at $\xi_0$ and $h(\xi_0)$ with respect to Newton-derivatives $h'$ and $g'$ in open sets $V$ and $U$, respectively, where $\xi_0 \in U\subset \mathbb{X}$, $h(U)\subset V \subset \mathbb{Y}$. Assume that $h$ is Lipschitz continuous in $U$ and $g'$ is locally bounded in $V$. Then $g\circ h : \mathbb{X} \to \mathbb{Z}$ is Newton-differentiable with respect to the Newton-derivative $g'(h(\cdot))h'(\cdot):U \to L(\mathbb{X},\mathbb{Z})$.
        \item[(ii)] Let $h$ and $g$ be strictly differentiable at $\xi_0$ and $h(\xi_0)$ with derivatives $h'(\xi_0)$ and $g'(h(\xi_0))$, respectively.
        Then $g\circ h : \mathbb{X} \to \mathbb{Z}$ is strictly differentiable with derivative $g'(h(\xi_0))h'(\xi_0) \in L(\mathbb{X},\mathbb{Z})$.
    \end{enumerate}
\end{lemma}
\begin{proof}
    For Newton-differentiabilty, a proof can be found in the lecture notes \cite{hintermuller2010semismooth}, for strict differentiability it is stated in \cite{nijenhuis1974strong}. 
\end{proof}
\begin{proposition}
    Consider $F : \X\to \Y$. 
    \begin{itemize}
     \item[i)] Let $F$ be locally Lipschitz continuous and $F'\in\Gamma(\L(T\X,F^*T\Y))$ be locally bounded. Then the criterion \eqref{eq:NewtonDiffbarVR} for Newton-differentiability is independent of the choice of charts. 
     \item[ii)]  The criterion \eqref{eq:StriktDiffbarVR} for strict differentiability of $F$ is independent of the choice of charts. 
    \end{itemize}
\end{proposition}\begin{proof}
    Let $\phi_i$ and $\psi_i$ be charts of $\X$ and $\Y$ such that $F$ is Newton-, respectively strictly, differentiable at $x_0\in\X$. Let $\phi_j$ be another chart of $\X$ and $\psi_j$ be another chart of $\Y$. Consider the mapping
    \[
        \psi_j\circ F \circ \phi_j^{-1} = t_{ji}^\Y \circ (\psi_i \circ F \circ \phi_i^{-1})\circ t_{ij}^\X.
    \]
    the changes of charts $t_{ij}^\X$ and $t_{ji}^\Y$ are continuously differentiable and thus Newton-differentiable, as well as strictly differentiable. By applying the chain rule twice, we obtain Newton- or strict differentiability of $\psi_j\circ F \circ \phi_j^{-1}$.
\end{proof}

 Using independence of chart representations, we can finally define \textit{Newton-differentiability} for mappings between manifolds:
\begin{definition}[Newton-differentiability of mappings between manifolds]
    Let $F:\X\to\Y$ be locally Lipschitz continuous and $F'\in\Gamma(\L(T\X,F^*T\Y))$ be locally bounded. 
    
    We call $F$ \emph{Newton-differentiable at a point $x_0\in\X$ with respect to $F'$} if all chart representations of $F$ and the corresponding representations of $F'$ satisfy \eqref{eq:NewtonDiffbarVR} for $\xi_0=\phi(x_0)$.
    In this case we call $F'$ a \emph{Newton-derivative} of $F$ at $x_0$.
\end{definition}

Analogously, we can define \textit{strict differentiability}: \begin{definition}[Strict differentiability of mappings between manifolds]
    Let $F:\X\to\Y$. $F$ is called \emph{strictly differentiable} at $x_0\in\X$, if its derivative $F'(x_0)\in L(T_{x_0}\X,T_{F(x_0)}\Y)$ satisfies \eqref{eq:StriktDiffbarVR} in all chart representations for $\xi_0 = \phi(x_0)$.
\end{definition}
Since the chart representation of a $C^1$-mapping $F:\X\to\Y$ is a $C^1$-mapping between linear spaces, we can conclude, that such $F$ is also Newton-differentiable and strictly differentiable. 
\section{Local convergence of Newton's method}
In this section we will study local convergence of Newton's method for mappings into vector bundles. This will require some preparation. Formulating local superlinear convergence requires a metric on $\X$ to measure the distance of iterates. In our general setting we are considering Banach manifolds which may not be endowed with a Riemannian metric. Thus, we will discuss a slightly more general way to define a metric on a manifold. Second, we will derive a geometric criterion on Newton- and strict differentiability. This will be especially useful later to connect our \emph{a-priori} analysis of local convergence with some computable \emph{a-posteriori} quantities, which are defined in the spirit of Deuflhard \cite{Deuflhard}. With these tools we can then give criteria for local superlinear convergence of Newton's method. 
\subsection{Local norms and a metric on manifolds}\label{sec:metric}
In Riemannian geometry, manifolds are equipped with a Riemannian metric, based on inner pro\-ducts on the tangent bundle \cite[VII §6]{Lang}. However, in the case of linear spaces, it is known that the convergence theory of Newton's method is not restricted to Hilbert spaces, but can be applied in general Banach spaces. Hence, we will present a slightly more general notion of a metric, replacing the usual inner products by general norms. In the infinite dimensional case, some care has to be taken, concerning continuity properties.  
\begin{definition}
Consider a function $f : \mathcal E \to \R$ on a vector bundle $p: \E\to \M$. We call $f$ \emph{fibrewise uniformly continuous} around $D \subset E_y$, $y\in\M$, if for any local trivialization $\tau$, the function 
\[
f_\tau:= f\circ \tau^{-1} : \tau(p^{-1}(U)) \to \R
\]
satisfies the following condition: \\
For all $\varepsilon > 0$ there exists a neighborhood $U\subset \M$ of $y$ and $\delta >0$, such that for all $(y,e_\tau) \in \tau(D)$ we have the estimate:
\[
 |f_\tau(y,e_\tau)-f_\tau(\hat y,\hat e_\tau)|< \varepsilon \quad \forall (\hat y,\hat e_\tau)\in p^{-1}(U) : \hat y \in U, \|\hat e_\tau-e_\tau\|_{\mathbb E} < \delta. 
\]
\end{definition}
It is not hard to show that this notion is independent of the choice of trivialization, taking into account that transition mappings are fibrewise linear isomorphisms. 

Clearly, fibrewise uniform continuity is a stronger condition than continuity, since $U$ and $\delta$ only depend on $y$ and not on $e_\tau$. In particular, we conclude, setting $\hat e_\tau =e_\tau$: 
\begin{equation}\label{eq:unicontsup}
 \sup_{e_\tau \in \tau(D)}|f_\tau(y,e_\tau)-f_\tau(\hat y,e_\tau)|< \varepsilon \quad \forall \hat y \in U.
\end{equation}
On a vector bundle $p : \E \to \M$ we can define a \emph{continuous norm} as follows:
\begin{definition}\label{def:contnorm}
 Consider a function $\|\cdot\| : \E \to [0,\infty)$ with the following properties:
 \begin{itemize}
  \item[i)] For each $y\in \M$ the restriction $\|\cdot\|_y$ to $E_y$ is a norm on $E_y$ which is equivalent to the $\Vert \cdot \Vert_{\mathbb{E}}$-norm that makes $\mathbb{E}$ a Banach space. Thus, $(E_y,\|\cdot\|_y)$ is a Banach space. We denote by $\mathbb{S}_y \subset E_y$ the unit sphere with respect to $\|\cdot\|_y$. 
  \item[ii)] For each $y\in \M$, $\|\cdot\|$ is fibrewise uniformly continuous around the unit sphere $\mathbb{S}_y \subset E_y$, where we denote
  \[
   \|e_\tau\|_{\tau,y}:=\|\cdot\|_\tau(y,e_\tau)=\|\tau_y^{-1}(e_\tau)\|_y
  \]

 \end{itemize}
We call $\|\cdot\|$ a \emph{continuous norm} on $\E$. 
\end{definition}
Our definition includes Riemannian metrics, as well as Finsler metrics, but is more general than those: no differentiability assumption is imposed, since this is not needed for our specific purpose. In particular, non-differentiable norms, such as $1$-norms or $\infty$-norms fit into our framework. 

\begin{example}
 Let $\mathbb E=L_q(\Omega)$ for $1\le q\le \infty$ and some measurabe set $\Omega$. Further let $s: \M \times \Omega \to (0,\infty)$ be a function, such that each $s(y,\cdot): \Omega \to (0,\infty)$ is measurable and bounded from above and below by positive constants (which may depend on $y$). Then for each $y\in \M$, $\|v\|_y :=\|s(y,\cdot)v\|_{L_q}$ defines a norm that is equivalent to $\|\cdot\|_{L_q}$. If we have in addition that $\lim_{\hat y\to y}\|s(y,\cdot)-s(\hat y,\cdot)\|_\infty=0$ for all $y\in \M$, then this norm is continuous in the sense of Definition~\ref{def:contnorm}, as can be seen by straightfoward application of the triangle inequality and the Hölder inequality. This is not the case, if we only assume that $s(\cdot,\omega)$ continuous for each $\omega \in \Omega$. 
\end{example}

\begin{lemma}
    \label{normEquivalence}
 Let $y\in \M$ and $\|\cdot\| : \E\to [0,\infty)$ be a continuous norm. Consider its representation $\|\cdot\|_\tau$ in a local trivialization around $y\in \M$. Then for every $\varepsilon >0$ there is a neighborhood $U$ of $y$, such that
 \[
  (1-\varepsilon)\|e_\tau\|_{\tau,y} \le \|e_\tau\|_{\tau,\hat y} \le (1+\varepsilon)\|e_\tau\|_{\tau,y}\quad \forall \hat y \in U, e_\tau\in \mathbb E. 
 \]
\end{lemma}
\begin{proof}
 Since $\|\lambda e_\tau\|_{\tau,y}=|\lambda|\|e_\tau\|_{\tau,y}$ we may assume $\|e_\tau\|_{\tau,y}=1$. Now our result is a direct consequence of \eqref{eq:unicontsup}.
\end{proof}
 To define a metric (in the sense of topology) on a manifold $\X$ we will consider a continuous norm on the tangent bundle $T\X$. The following integral over a (piecewise) $C^1$-curve $\alpha : [a,b] \rightarrow \X$ is well defined, because the integrand is (piecewise) continuous on $[a,b]$:
\begin{equation*}
    L(\alpha) := \int_a^b \|\alpha^\prime (t)\|_{ \alpha (t)} \, dt. 
\end{equation*}
The value of this integral can be interpreted as the length of the curve $\alpha$. The following derivations are analogous to the construction of a Riemannian distance (cf. e.g. \cite[VII §6]{Lang}), and generalize them from the case of a Hilbert space norm to the case of a Banach space norm. 

To measure the \grqq distance\grqq \ between $x$ and $y$ on $\X$, we form the infimum over the lengths of all curves connecting $x$ and $y$.
If $x$ and $y$ cannot be connected, we set $d(x,y) = \infty$. According to this we define the map
\begin{align*}
    d : \X \times \X &\rightarrow \overline{\R} := \R \cup \lbrace \infty \rbrace \\
    (x,y) &\mapsto \inf \lbrace L(\alpha) \ | \ \alpha:[a,b] \rightarrow \X \text{ piecewise $C^1$-curve s.t. } \alpha (a) = x, \alpha (b) = y \rbrace.
\end{align*}
In the following we will always assume that $\X$ is connected. Then most of the axioms of a metric follow readily. Symmetry $d(x,y)=d(y,x)$  and non-negativity are obvious from the definition. It is known from topology that on a connected manifold, two points can always be connected by a piecewise smooth curve \cite{topology}, and thus $d(x,y) < \infty \ \forall x,y \in \X$. The triangle inequality follows straightforwardly from the fact that the length of curves adds up if they are concatenated.

It thus remains to show positive definiteness of $d$, which is the only property of a metric that requires fibrewise uniform continuity. 
To this end, we first prove some results which relate the metric to the  norm, locally. 
Let $(U,\phi)$ be a local chart at $x_0\in \X$ and $\Phi: \pi^{-1}(U) \to U\times \mathbb{X}$ be the corresponding natural vector bundle chart on the tangent bundle $\pi : T\X \to \X$.  We define the \textit{open ball with radius} $r>0$ \textit{around} 0 \textit{with respect to the local norm} $\Vert\cdot \Vert_{x_0}$ (with representation $\Vert\cdot\Vert_{\Phi, x_0}$) as 
\[
    rB_{x_0} := \{ v \in \mathbb X : \Vert v\Vert_{\Phi,x_0} < r \}.
\]
Note that in this special case the model space $\mathbb{E}$ of the vector bundle coincides with the model space $\mathbb{X}$ of the manifold. In particular, applying $\Vert\cdot\Vert_{\Phi, x_0}$ to chart representations of elements of $\X$ is well defined.
\begin{lemma}
    \label{lemma:normMetrik}
    Let $x_0 \in \X$ and $(U,\phi)$ a local chart, such that  $\phi(x_0) = 0$. Let $x,y \in U$ with representatives $x_\phi$ and $y_\phi$ in the chart. Let $\varepsilon >0$. Then there exists a radius $r>0$ such that for every $C^1$-curve $\alpha :[a,b] \to \X$ connecting $x$ and $y$ the following holds:
    \begin{enumerate}
    \item[(a)] If the image $\alpha_\phi : [a,b]\to \mathbb{X}$ lies in the open ball $rB_{x_0}$, in particular $x_\phi, \; y_\phi \in rB_{x_0}$, then the following properties hold:
    \begin{enumerate}
        \item[(i)] $(1-\varepsilon) \Vert x_\phi - y_\phi\Vert_{\Phi, x_0} \leq L(\alpha)$
        \item[(ii)] $d(x,y) \leq (1+\varepsilon)\Vert x_\phi - y_\phi\Vert_{\Phi,x_0}$
    \end{enumerate}
    \item[(b)] If $\alpha_\phi$ is leaving the ball $rB_{x_0}$, i.e. $\alpha_\phi(a)=x_\phi \in rB_{x_0}$ but $\alpha_\phi \not\subset rB_{x_0}$, then it holds that
    \[
        L(\alpha) \geq (1-\varepsilon)(r - \Vert x_\phi\Vert_{\Phi, x_0}).
    \]
\end{enumerate}
\end{lemma}
\begin{proof}
    Since $\Vert \cdot \Vert$ is a continuous norm on $T\X$, according to Lemma \ref{normEquivalence} for $ \varepsilon>0$ there exists a neighborhood $V$ of $x_0$, such that 
    \[
        (1-\varepsilon)\|v\|_{\Phi,x_0} \le \|v\|_{\Phi,\xi} \le (1+\varepsilon)\|v\|_{\Phi,x_0}\; \forall \xi \in V, v \in \mathbb X.
    \]
    The image of $V$ in the chart contains a $\|\cdot\|_{\Phi,x_0}$-ball $r B_{x_0}$ around $\phi(x_0) = 0$ with radius $r = r(\varepsilon)>0$. 
    Let $x,y\in U$ and $\alpha: [a,b]\to\X$ be a $C^1$-curve connecting $x$ and $y$.
    \begin{enumerate}
        \item[(a)] Assume that the image $\alpha_\phi$ lies in $rB_{x_0}$. Using that the curve integral in vector spaces is always greater than or equal to the distance of the end points, we can estimate the length of $\alpha$ by
    \begin{align*}
        (1-\varepsilon)\Vert x_\phi - y_\phi\Vert_{\Phi,x_0}&\leq (1-\varepsilon) \int_a^b \Vert\alpha_\phi^\prime(t)\Vert_{\Phi, x_0} \; dt \overset{\alpha(t)\in V}{\leq} \int_a^b \Vert \alpha^\prime(t)\Vert_{\alpha(t)} \; dt \\
        &= L(\alpha) \leq (1+\varepsilon) \int_a^b \Vert\alpha_\phi^\prime(t)\Vert_{\Phi, x_0} dt.
    \end{align*}
    Choosing the connecting line of $x_\phi$ and $y_\phi$ for $\alpha_\phi$ this yields $d(x,y) \leq (1+\varepsilon) \Vert x_\phi - y_\phi\Vert_{\Phi,x_0}$.
    
    \item[(b)] If $\alpha_\phi$ leaves the open ball $rB_{x_0}$, there must be a first intersection point $s_\phi = \alpha_\phi(c)$, $c\in[a,b]$, with the sphere $r\mathbb S_{x_0}\subset \mathbb{X}$ around $0$. Let $\widetilde{\alpha}_\phi$ be the part of $\alpha_\phi$ connecting $x_\phi$ and $s_\phi$. Then by the norm equivalence and the inverse triangle inequality:
    \begin{align*}
        L(\widetilde{\alpha}) = \int_a^c \Vert \widetilde{\alpha}^\prime(t)\Vert_{\widetilde{\alpha}(t)} dt &\geq (1-\varepsilon) \int_a^c\Vert \widetilde{\alpha}_\phi^\prime(t)\Vert_{\Phi, x_0} dt \\
        &\geq (1-\varepsilon)(\Vert s_\phi\Vert_{\Phi,x_0} - \Vert x_\phi\Vert_{\Phi, x_0}) = (1-\varepsilon)(r-\Vert x_\phi\Vert_{\Phi, x_0}).
    \end{align*}
    Thus, we get $L(\alpha) \geq (1-\varepsilon)(r-\Vert x_\phi\Vert_{\Phi, x_0})$.
\end{enumerate}
\end{proof}
\begin{proposition}
    \label{prop:normMetrik}
    Let $x_0 \in \X$ and $(U,\phi)$ a local chart and
    let $\varepsilon>0$.
Then there is a neighborhood $W\subseteq U$ of $x_0$ such that
    \begin{equation*}
        (1-\varepsilon)\,\|x_\phi - y_\phi \|_{\Phi, x_0} \leq d(x,y) \leq (1+\varepsilon) \, \|x_\phi - y_\phi \|_{\Phi, x_0} \ \forall x,y \in W,
    \end{equation*}
    where $x_\phi, y_\phi \in \mathbb{X}$ are representatives of $x$ and $y$ in the chart.
\end{proposition}
\begin{proof}
W.l.o.g. assume $\phi(x_0) = 0$. By Lemma \ref{lemma:normMetrik}, we directly obtain
\[
    d(x,y) \leq (1+\varepsilon) \|x_\phi - y_\phi \|_{\Phi, x_0} \; \forall x,y \in \phi^{-1}(rB_{x_0}),
\]
where $r>0$ is chosen sufficiently small. Now let $\alpha$ be an arbitrary curve connecting $x$ and $y$. Lemma \ref{lemma:normMetrik} yields
\begin{align*}
    \alpha_\phi \subset rB_{x_0} &\Rightarrow L(\alpha) \geq (1-\varepsilon)\Vert x_\phi - y_\phi\Vert_{\Phi, x_0}, \\
    \alpha_\phi \not\subset rB_{x_0} &\Rightarrow L(\alpha) \geq (1-\varepsilon)(r-\Vert x_\phi\Vert_{\Phi, x_0}).
\end{align*}
Thus, we obtain
\begin{equation}
    \label{minAbschMetrik}
    d(x,y) \geq (1-\varepsilon)\cdot \min\{\Vert x_\phi - y_\phi\Vert_{\Phi, x_0}, r-\Vert x_\phi \Vert_{\Phi, x_0}\}.
\end{equation}
Choosing $W := \phi^{-1}(\frac{r}{3}B_{x_0})$ yields $d(x,y) \geq (1-\varepsilon) \Vert x_\phi - y_\phi\Vert_{\Phi, x_0}$ and thus our claim.
\end{proof}
\begin{corollary}
    Let $\X$ be a connected $C^1$-Banach manifold with the Hausdorff property. Then $d$ defines a metric on $\X$. Moreover, the induced topology coincides with the topology on $\X$ induced by the atlas.
\end{corollary}
\begin{proof}
    Let $x \neq y$. Let $(U,\phi)$ be a local chart at $x$ such that w.l.o.g. $\phi (x) = x_\phi = 0$, and let $\Phi : \pi^{-1}(U) \rightarrow U \times \mathbb{X}$ be the natural vector bundle chart on the tangent bundle $\pi:T\X \rightarrow \X$.\\
    Since $\X$ satisfies the Hausdorff property and $x\neq y$ holds, there is a $\|\cdot\|_{\mathbb X}$-ball, and by equivalence of norms also a
    $\|\cdot\|_{\Phi, x}$-ball $rB_{x}$ around $0$ with radius $r>0$ such that $y \not\in \phi^{-1}(rB_{x})$. Possibly decreasing $r$ first, we obtain by (\ref{minAbschMetrik})
    \[
        d(x,y) \geq (1-\varepsilon) \cdot \min\{\Vert y_\phi\Vert_{\Phi, x}, r\} = (1-\varepsilon)\cdot r > 0.
    \]
    Moreover, for $y \in \phi^{-1}(r\mathbb{S}_{x})$ we obtain by Prop. \ref{prop:normMetrik}
    \[
        (1-\varepsilon)\cdot \Vert y_\phi\Vert_{\Phi, x} \leq d(x,y) \leq (1+\varepsilon) \cdot \Vert y_\phi\Vert_{\Phi, x}.
    \]
    Since $\Vert \cdot \Vert_{\Phi, x}$ is equivalent to $\Vert \cdot \Vert_{\mathbb{X}}$, this implies that every $\mathbb{X}$-ball contains a $d$-ball and vice versa. Thus, the topology induced by $d$ coincides with the topology induced by the atlas of $\X$.
\end{proof}
Now consider a retraction $R :T\X \rightarrow \X$. We will relate norms of its preimages on the tangent spaces, to the distance of its images on $\X$, measured on the corresponding metric. First of all, we establish existence and smoothness of an inverse:

\begin{proposition}
\label{inverseRetr}
    Let $\X$ be a $C^1$-Banach manifold, $x_0 \in \X$ and let $R$ be a $C^1$-retraction. 
    Let $(V, \phi)$ be a chart of $\X$ at $x_0$ and $\Phi$ be the corresponding natural tangent bundle trivialization. 
    
    Then for any $\varepsilon >0 $ there exists a neighborhood $U\subset V$ of $x_0$, such that the following holds:
    \begin{itemize}
     \item[i)]  A local inverse $R_z^{-1} : U \to T_z\X$ exists for each $z\in U$. Denote their chart representation by 
    \begin{equation}
        \label{eq:inverseRetractionRepresentation}
        R_{z,\phi}^{-1} : \phi(U) \to \mathbb{X}.
    \end{equation}    
    \item[ii)] For all $x, y, z \in U$ we have the estimate:
    \begin{equation}
        \label{eq:strictDifferentiabilityInverseRetraction}
        \Vert x_\phi-y_\phi - (R_{z,\phi}^{-1}(x_\phi)-R_{z,\phi}^{-1}(y_\phi))\Vert_{\Phi, x_0} \leq \varepsilon \Vert x_\phi - y_\phi \Vert_{\Phi, x_0}.
    \end{equation}
    \end{itemize}
\end{proposition}
\begin{proof}
    Application of the inverse function theorem to the chart representation $\widehat R_\phi$ of the mapping
    \begin{align*}
        \widehat{R} : T\X &\to \X \times \X \\
        (z,v) &\mapsto (z, R_z(v))
    \end{align*}
    yields that $\widehat R_\phi : \phi(V)\times \mathbb X\to \mathbb X \times \mathbb X$ is a local diffeomorphism, and it holds $(z_\phi,R_{z,\phi}^{-1}(x_\phi))=\widehat R_\phi^{-1}(z_\phi,x_\phi)$.
    Since $\widehat R_\phi^{-1}$ is continuously differentiable and thus strictly differentiable we obtain a neighborhood $W \subset \mathbb X\times \mathbb X$ of $(x_{0, \phi},x_{0, \phi})$, such that for $(z_\phi,x_\phi),(z_\phi,y_\phi)\in W$ we have, using $(z_\phi,x_\phi)-(z_\phi,y_\phi)=(0,x_\phi-y_\phi)$ and denoting by $[v]_2$ the second component of $v\in \mathbb X\times \mathbb X$:
        \begin{equation}
        \label{eq:strictDifferentiabilityInverseRetractionInCharts}
        \Vert [(\widehat R_\phi^{-1})'(x_{0, \phi},x_{0, \phi})(0,x_\phi-y_\phi) - (\widehat R_\phi^{-1}(z_\phi,x_\phi)-\widehat R_\phi^{-1}(z_\phi,y_\phi))]_2\Vert_{\Phi, x_0} \leq \varepsilon \Vert x_\phi - y_\phi \Vert_{\Phi, x_0}.
    \end{equation}
    Taking into account that $(R_{x_0,\phi}^{-1})'(x_{0, \phi}) = (R_{x_0,\phi}'(R_{x_0,\phi}^{-1}(x_0, \phi)))^{-1} = Id_{\mathbb{X}}$, this in turn implies \eqref{eq:strictDifferentiabilityInverseRetraction}, if we choose $U$, such that for $x,y,z\in U$ we get $(z_\phi,x_\phi),(z_\phi,y_\phi)\in W$.
\end{proof}

\begin{proposition}
    \label{prop:RetraktionMetric}
    To $x_0 \in \X$ and $\varepsilon>0$ there exists a neighborhood $U$ of $x_0$ such that for all $x,y,z \in U$ the following holds:
    \begin{equation*}
        (1-\varepsilon) \|R^{-1}_z(x) - R^{-1}_z(y)\|_z \leq d(x,y) \leq (1+\varepsilon) \|R^{-1}_z(x) - R^{-1}_z(y)\|_z.
    \end{equation*}
\end{proposition}
\begin{proof}
    Choose a chart $(V,\phi)$ of $\X$ at $x_0$ and denote by $\Phi: \pi^{-1}(V) \to \mathbb{X}\times \mathbb{X}$ the natural vector bundle chart on the tangent bundle $\pi: T\X \to \X$. Let $x_{0,\phi}, \ x_\phi, \ y_\phi$ and $z_\phi$ be representatives of $x_0,x,y$ and $z$ in the chart. Let $R_{z, \phi}^{-1}$ be the chart representation of $R_z^{-1}$ according to \eqref{eq:inverseRetractionRepresentation}. We show that
    \begin{equation*}
        \lim_{x_\phi, y_\phi,z_\phi \rightarrow x_{0,\phi}}\frac{\|R_{z, \phi}^{-1}(x_\phi) - R_{z, \phi}^{-1}(y_\phi)\|_{\Phi, x_0}}{\|x_\phi - y_\phi\|_{\Phi, x_0}} = 1.
    \end{equation*}
    Using \eqref{eq:strictDifferentiabilityInverseRetraction} and the inverse triangle inequality, we obtain that for any $\varepsilon >0$ there exists a neighborhood $V$ of $x_0$, such that
    \begin{align*}
        \big|\|x_\phi-y_\phi\|_{\Phi,x_0}-\Vert R_{z,\phi}^{-1}(x_\phi) - R_{z,\phi}^{-1}(y_\phi)\Vert_{\Phi, x_0}\big| &\le \varepsilon\Vert x_\phi - y_\phi\Vert_{\Phi, x_0}.
    \end{align*}
    By possibly shrinking $V$,
    we obtain the desired assertion by Prop. \ref{prop:normMetrik}.
\end{proof}
\subsection{Newton- and Strict Differentiability in Geometric Form}
In order to analyse local convergence of Newton's method, we want to prove a result about a Newton- or strictly differentiable mapping $F:\X\to\E$ between a manifold $\X$ and a vector bundle $p:\E\to\M$, which can be stated independently of the choice of charts.

In the following, we assume that the mapping 
\[
    y = p\circ F : \X \to \M, \; y(x) := p(F(x))\in \M,
\]
which computes the base point of $F(x)\in\E$ for any $x\in \X$, is locally Lipschitz continuous. We can use the linear map $\VT{y}{\leftarrow}(y_0) \in L(E_{y_0},E_{y})$ given by a vector back-transport $\VT{y}{\leftarrow} \in \Gamma(\L(\E, E_{y}))$ to transport elements $F(x_0) \in E_{y_0}$ into the fibre $E_y$. Such a vector back-transport may be consistent to a connection map $Q \in \Gamma(\L(T\E, p^*\E))$, i.e., according to Def. \ref{def:consistentConnection}, we may have $Q_e = \langle \VT{y}{\leftarrow}\rangle'(e), \; e\in\E$. 
\begin{proposition}
    \label{QVNewtondiffbar}
    Let $F:\X \to \E$. Let $\VT{y}{\leftarrow}\in\Gamma(\L(\E, E_y))$ be a vector back-transport, $Q\in \Gamma(\L(T\E,p^*\E))$ be a connection map and $R:T\X\to\X$ be a $C^1$-retraction.
    \begin{enumerate}
        \item[(i)] \label{bruch}
        Let $F : \X \to \E$ be Newton-differentiable at $x_0\in\X$ with respect to $F^\prime \in \Gamma(\L(T\X,F^*T\E))$. If $\VT{y}{\leftarrow}$ is consistent with $Q$ or $F(x_0)=0$, then it holds
        \begin{equation*}
            \underset{x\to x_0}{\lim} \frac{\Vert(Q_{F(x)} \circ F^\prime(x)) (R_x^{-1}(x) - R_x^{-1}(x_0)) - (F(x) - \VT{y(x)}{\leftarrow}(y(x_0))F(x_0))\Vert_{E_{y(x)}}}{d_\X(x,x_0)} = 0.
        \end{equation*}
    \item[(ii)]
        Let $F$ be strictly differentiable at $x_0 \in \X$. If $\VT{y}{\leftarrow}$ is consistent with $Q$ or $F(x_0) = 0$ then
        for every $\varepsilon >0$ there exists a neighborhood $W$ of $x_0$ such that for all $x,\xi \in W$ it holds
        \begin{align*}
            \Vert (Q_{F(x_0)} \circ F'(x_0))(R_{x_0}^{-1}(x) &- R_{x_0}^{-1}(\xi)) - (\VT{y(x_0)}{\leftarrow}(y(x))F(x) - \VT{y(x_0)}{\leftarrow}(y(\xi))F(\xi))\Vert_{E_{y(x_0)}} \\
            &< \varepsilon \Vert R_{x_0}^{-1}(x) - R_{x_0}^{-1}(\xi) \Vert_{x_0}.
        \end{align*}
    \end{enumerate}
\end{proposition}
\begin{proof}
Choose a chart $(U, \phi)$ of $\X$ and a vector bundle chart $\eta$ of $\E$. 
Let $u,v,w\in U$. In the following, we use $\phi$ or $\eta$ as an index to denote the chart representation of elements of $\X$ or $\E$, respectively.
We begin by discussing the numerators in $(i)$ and $(ii)$ in a unified way and collecting all remainder terms which occur. Then, for the proof of $(i)$ we set $u=v=x$ and fix $w=x_0$. For the proof of $(ii)$ we will fix $u=x_0$ and consider $v=x, w=\xi$. With these notations the remainder terms read in both cases
\[
 r(u,v,w)=(Q_{F(u)} \circ F'(u))(R_{u}^{-1}(v) - R_{u}^{-1}(w)) - \big(\VT{y(u)}{\leftarrow}(y(v))F(v) - \VT{y(u)}{\leftarrow}(y(w))F(w)\big).
\]
Let $R_{u, \phi}^{-1}$ be a representation of the inverse retraction in charts given by \eqref{eq:inverseRetractionRepresentation}. Since $R_u^{-1}$ is a $C^1$-mapping with $(R_u^{-1})^\prime(u) =R_u'(R_u^{-1}(u))^{-1}= Id_{T_{u}\X}$, we obtain
\[
         R_{u, \phi}^{-1}(v_{\phi}) - R_{u, \phi}^{-1}(w_{\phi}) = (R^{-1}_{u, \phi})^\prime(u_\phi)(v_\phi - w_{\phi}) + r_{R^{-1}}  = (v_\phi - w_{\phi}) + r_{R^{-1}}     
\]
    where $r_{R^{-1}}$ is a remainder term. 
   Setting $\delta y_\eta := y(v)_\eta- y(w)_{\eta} \in \mathbb{Y}$, we obtain the chart representation of 
    \begin{equation*}\label{eq:zaehler1}
        (Q_{F(u)} \circ F^\prime(u))(R_{u}^{-1}(v) - R_{u}^{-1}(w))
    \end{equation*}
    by
    \begin{equation}
        \label{QFKarte}
        (F_{\eta,\phi})^\prime(u_\phi)(v_\phi - w_{\phi}+r_{R^{-1}}) - B_{y(u),\eta}(F_{\eta,\phi}(u_\phi))\delta y_\eta.
    \end{equation}
Next, we consider the representation of 
$\VT{y(u)}{\leftarrow}(y(v))F(v) - \VT{y(u)}{\leftarrow}(y(w))F(w)$,
which is of the form
\begin{equation}
    \label{zaehlerRepr}
    \VT{y(u),\eta}{\leftarrow}(y(v)_\eta)F_{\eta,\phi}(v_\phi) - \VT{y(u),\eta}{\leftarrow}(y(w)_\eta)F_{\eta,\phi}(w_\phi).
\end{equation}
Using the differentiability of the mapping $\VT{y(u),\eta}{\leftarrow} : \mathbb Y \to L(\mathbb E, \mathbb E)$,  we can write:
\begin{equation*}
    \VT{y(u),\eta}{\leftarrow}(y(v)_{\eta}) - \VT{y(u),\eta}{\leftarrow}(y(w)_{\eta})=\VTprime{{y(u),\eta}}{\leftarrow}(y(u)_\eta)\delta y_\eta + r_V
\end{equation*}
where $r_V$ denotes a remainder term. Inserting this into \eqref{zaehlerRepr} we obtain by linearity of $\VT{y(u)}{\leftarrow}(y(v))$ 
\begin{align}\label{eq:zaehler2rep}
    \eqref{zaehlerRepr}=&\;(\VTprime{{y(u),\eta}}{\leftarrow}(y(u)_\eta)\delta y_\eta)F_{\eta,\phi}(w_\phi)+\VT{{y(u),\eta}}{\leftarrow}(y(v)_{\eta})(F_{\eta,\phi}(v_\phi)-F_{\eta,\phi}(w_\phi))  + r_V F_{\eta,\phi}(w_\phi).
\end{align}
Writing the difference $F_{\eta,\phi}(v_\phi)-F_{\eta,\phi}(w_{\phi})$ as
\begin{equation*}\label{eq:remF}
    F_{\eta,\phi}(v_\phi)-F_{\eta,\phi}(w_{\phi}) = (F_{\eta,\phi})^\prime(u_\phi)(v_\phi - w_{\phi}) + r_F
\end{equation*}
and subtracting \eqref{eq:zaehler2rep} from \eqref{QFKarte}  we finally obtain the following representation of the numerators of $(i)$ and $(ii)$ in charts: 
\begin{equation}
    \label{eq:denominatorRepr}
    \begin{aligned}
        r(u,v,w)\sim r_{\eta,\phi}(u_\phi,v_\phi,w_\phi)&=(Id_{E_{y(u)}} - \VT{{y(u),\eta}}{\leftarrow}(y(v)_\eta))(F_{\eta,\phi})^\prime(u_\phi)(v_\phi - w_{\phi})\\
        &- B_{y(u),\eta}(F_{\eta,\phi}(u_\phi))\delta y_\eta - (\VTprime{{y(u),\eta}}{\leftarrow}(y(u)_\eta)\delta y_\eta)F_{\eta,\phi}(w_\phi) \\
        &+(F_{\eta,\phi})^\prime(u_\phi) r_{R^{-1}}-\VT{{y(u),\eta}}{\leftarrow}(y(v)_{\eta})r_F - r_VF_{\eta,\phi}(w_\phi).
    \end{aligned}
\end{equation}
We will discuss the three lines of \eqref{eq:denominatorRepr} separately.

In the Newton-differentiable case, where $u=v=x$ we have $\VT{{y(u)}}{\leftarrow}(y(v)) = Id_{E_y}$ and thus, the first line vanishes. In the strictly differentiable case $(F_{\eta,\phi})'(u_\phi)$ is fixed. Thus, we obtain that $(F_{\eta,\phi})^\prime(u_\phi)(v_\phi - w_{\phi})$ is of order $\mathcal{O}(\Vert v_\phi - w_\phi\Vert_{\mathbb{X}})$. 
Using that
$\VT{{y(u),\eta}}{\leftarrow}(y(v)_\eta) \to Id_{E_{y(u)}}$ for $y(v) \to y(u)$, and $x\mapsto y(x)$ is continuous, we get that $Id_{E_{y(u)}} - \VT{{y(u),\eta}}{\leftarrow}(y(v)_{\eta})$ tends to $0$ in the operator norm as $v\to u$. Thus, in both cases the first line of \eqref{eq:denominatorRepr} vanishes in the desired order for all choices of $u,v,w$ corresponding to the respective differentiability concepts.

Next, we have to discuss the second line of \eqref{eq:denominatorRepr}
 using our consistency assumption or $F(x_0)=0$.
If $\VT{y}{\leftarrow}$ is consistent with the connection map $Q$ we have
    \[
        B_{y(u),\eta}(F_{\eta,\phi}(u_\phi))\delta y_\eta = - (\VTprime{{y(u),\eta}}{\leftarrow}(y(u)_\eta)\delta y_\eta)F_{\eta,\phi}(u_\phi).
    \]
Thus, in this case the second line of \eqref{eq:denominatorRepr} simplifies to
    \begin{equation}\label{eq:Ncase1}
        (\VTprime{{y(u),\eta}}{\leftarrow}(y(u)_\eta)\delta y_\eta)(F_{\eta,\phi}(u_\phi) - F_{\eta,\phi}(w_\phi)).
    \end{equation}
Alternatively, we assume that $F(x_0)=0$. In the Newton differentiable case this yields $F_{\eta,\phi}(w_\phi) = 0$, in the strictly differentiable case we have that $F_{\eta,\phi}(u_\phi) = 0$. Thus, we can rewrite the second line of \eqref{eq:denominatorRepr} in the Newton-differentiable case as 
\begin{equation}\label{eq:Ncase2}
    -B_{y(u),\eta}(F_{\eta,\phi}(u_\phi)-F_{\eta,\phi}(w_\phi))\delta y_\eta,
    \end{equation}
and in the strictly differentiable case as
\begin{equation}\label{eq:Ncase3}
    (\VTprime{{y(u),\eta}}{\leftarrow}(y(u)_\eta)\delta y_\eta)(F_{\eta,\phi}(u_\phi) - F_{\eta,\phi}(w_\phi)).
    \end{equation}
In all the above cases \eqref{eq:Ncase1}-\eqref{eq:Ncase3}, using Lipschitz continuity of $F_{\eta,\phi}$, bilinearity of $\VTprime{{y(u),\eta}}{\leftarrow}(y(u)_\eta)$ or of $B_{y(u),\eta}$, and Lipschitz continuity of the mapping $x \mapsto y(x)$, implying $\delta y_\eta=\mathcal O(\|v_\phi-w_\phi\|_{\mathbb X})$, the second line also vanishes in the desired order.

Finally, we discuss the third line of \eqref{eq:denominatorRepr}, where the remainder terms $r_{R^{-1}}$, $r_V$, and $r_F$ vanish in both cases in the desired order, since the mappings $R_u^{-1}$ and $\VT{y}{\leftarrow}$ are continuously differentiable and $F$ is Newton- respectively strictly differentiable.
In addition $(F_{\eta,\phi})'(u_\phi)$ is locally bounded by assumption in the Newton-differentiable case and fixed in the strict differentiable case, $\VT{{y(u),\eta}}{\leftarrow}(y(v)_{\eta})$ is locally bounded  and $F_{\eta,\phi}(w_\phi)$ is locally bounded since $w$ is fixed in the Newton-differentiable case and by continuity in the strictly differentiable case.  
Thus, the third line of \eqref{eq:denominatorRepr} also vanishes in the desired order if we choose $u,v,w$ corresponding to the respective differentiability concepts.

\vspace{0.1cm} 
In summary, we get the following smallness result. For each $\varepsilon > 0$ there exists a neighbourhood $U_\varepsilon$ of $x_0$, such that
\[
 \|r_{\eta,\phi}(u_\phi,v_\phi,w_\phi)\|_{\mathbb E} \le \varepsilon \|v_\phi-w_\phi\|_{\mathbb X}  
\]
for all choices of $u,v,w\in U_\varepsilon$ corresponding to the respective differentiability concept. 
By the equivalences of norms, established in Prop. \ref{prop:normMetrik} and Prop. \ref{prop:RetraktionMetric} in addition, we obtain the desired results. 
\end{proof}
\begin{bemerkung}
    For the special case of vector fields on Riemannian manifolds $\X$ Newton-differentiability was defined in \cite[Def. 18]{de2020newton} by using the exponential map $\exp : T\X \to \X$ and parallel transports along geodesics. In our notation their definition for semismoothness of a mapping $F\in \Gamma(T\X)$ at $x_0 \in \X$ w.r.t. a Newton derivative $Q_{F(x)} \circ F'(x) \in L(T_x\X, T_x\X)$ reads
    \[
        \|F(x_0) - P_{x_0\leftarrow x}[F(x)+Q_{F(x)}\circ F'(x)\exp^{-1}_x x_0]\|_x \leq \epsilon d_\X(x,x_0)
    \]
    where $P_{x_0\leftarrow x} : T_{x}\X \to T_{x_0}\X$ denotes the parallel transport. By applying the parallel transport {$P_{x\leftarrow x_0}: T_{x_0}\X \to T_{x}\X$} this can equivalently stated as 
    \[
        \|Q_{F(x)}\circ F'(x)(- \exp_x^{-1}x_0) - (F(x) - P_{x \leftarrow x_0}F(x_0))\|_x \leq \epsilon d_\X(x,x_0)
    \]
    which coincides with Newton-differentiability in geometric form (cf. Prop. \ref{QVNewtondiffbar} (i) with $\E=T\X$, $y = id$) by choosing $\VT{y(x)}{\leftarrow}(y(x_0)) = P_{x\leftarrow x_0}$ and $R_x = \exp_x$, and using $R_x^{-1}(x) = 0_x$.
\end{bemerkung}
\subsection{Local superlinear convergence}
In this section we want to prove the local superlinear convergence of Newton's method. In the spirit of Deuflhard \cite{Deuflhard}, our proof will rely on quantities, which not only allow a qualitative \emph{a-priori} convergence result, but also for which good \emph{a-posteriori} algorithmic estimates are accessible. This second aspect will be the basis of an affine covariant damping strategy, elaborated in the next section. In contrast to \cite{Deuflhard}, which relies on affine covariant Lipschitz constants, we take a more qualitative approach and consider affine covariant versions of Newton- and strict differentiabililty, as established in Proposition~\ref{QVNewtondiffbar}.  

Consider again $F:\X\to \E$ where the $C^1$-Banach manifold $\X$ is equipped with a continuous local norm $\|\cdot\|_x$ and an induced metric $d_\X$. 
Let $\VT{y}{\leftarrow} \in \Gamma(\L(\E,E_y))$ be a vector back-transport, $Q \in \Gamma(\L(T\E,p^*\E))$ be a connection map, and $R:T\X\to\X$ be a $C^1$-retraction. For a fixed $x\in\X$ we can locally perform \emph{pullbacks}
\begin{align*}
    R_x^{-1} : \X &\to T_x\X \\
                \hat{x} &\mapsto \hat{\textbf{x}} := R_x^{-1}(\hat{x}).
\end{align*}
Since $R$ is a retraction, we obtain $\textbf{x} = R_x^{-1}(x) = 0_x$. For the Newton method, defined in  Section~\ref{sec:DefNewton},
we consider the following \textit{affine-covariant quantity} at $x\neq z \in \X$
    \begin{equation*}
        \theta_{z}({x}) := \frac{\big\|(Q_{F(x)} \circ F^\prime(x))^{-1} \big[(Q_{F(x)} \circ F^\prime(x)) (\textbf{x}-\textbf{z}) - (F(x) - \VT{{y(x)}}{\leftarrow}(y(z))F(z)) \big]\big\|_{x}}{\|\textbf{x}-\textbf{z} \|_{x}}.
    \end{equation*}
The use of this quantity gives us a very simple result on local superlinear convergence of Newton's method, which serves two purposes. First, it will be the basis for an a-priori result for Newton-differentiable mappings, second, it yields an algorithmic idea to monitor local convergence.
\begin{proposition}\label{pro:localconvergenceAI}
    Let $\X$ be a manifold and $p:\E\to\M$ a vector bundle. Consider a mapping $F:\X\to\E$ and a section $F'\in \Gamma(\mathcal L(T\X,F^*T\E))$. Let $Q \in \Gamma(\L(T\E,p^*\E))$ be a connection map and $R:T\X\to\X$ be a retraction. Let $x_\star\in\X$ be a zero of $F$ and assume that all Newton steps $x_{k}$ are well defined. Assume that
    \begin{equation}
        \label{thetaLim}
        \lim_{x\to x_\star} \theta_{x_\star}(x) = 0.
    \end{equation}
Then Newton's method converges locally to $x_\star$ with a superlinear rate, i.e.
\[
    d_\X(x_\star,x_{k+1}) \leq \theta_{x_\star}(x_k)d_\X(x_\star,x_k).
\]
\end{proposition}
\begin{proof}
Let $x\in \X$. Set $x_+ := R_x(\delta x)$, where $\delta x \in T_x\X$ is the Newton direction at $x$.
Using the pullback $\textbf{x}_{+} = R_{x}^{-1}(x_{+}) = \delta x$ and $F(x_\star)=0$ we get the following equation in the tangent space $T_{x}\X$:
\begin{equation*}
    \| \textbf{x}_{+} - \textbf{x}_\star \|_{x} = \| (Q_{F(x)} \circ F^\prime(x))^{-1} \big[ (Q_{F(x)}\circ F^\prime(x)) (\textbf{x} - \textbf{x}_\star) - (F(x) - \VT{{y(x)}}{\leftarrow}(y(x_\star))F(x_\star))\big]\|_{x}.
\end{equation*}
By definition of $\theta_{x_\star}(x)$ it follows
\begin{equation}
    \label{normenNewton}
    \| \textbf{x}_{+} - \textbf{x}_\star \|_{x} = \theta_{{x}_\star}({x}) \cdot \|\textbf{x} - \textbf{x}_\star \|_{x}.
\end{equation}
Consider a neighborhood $V$ of $x_\star$, where Prop. \ref{prop:RetraktionMetric} holds for some $\varepsilon > 0$. Since $\displaystyle \lim_{x\to x_\star}\theta_{x_\star}(x)=0$, we find a metric ball $r\B_{x_\star}=\{x\in \X : d_\X(x,x_\star)<r\} \subset V$ of radius $r>0$, such that
\[
    \theta_{x_\star}(x) \leq \frac12\frac{1-\varepsilon}{1+\varepsilon}  \ \forall x \in r\B_{x_\star}.
\]
Now choose a starting point $x_0 \in r\B_{x_\star}$ and consider the sequence of Newton steps given by $x_{k+1} := x_{k,+}$ for $k \ge 0$. Assuming for induction that $x_k \in 2^{-k}r\B_{x_\star}$, we obtain by using \eqref{normenNewton}:
\begin{align*}
    d_\X(x_{k+1},x_\star) &\leq \theta_{x_\star}(x_k)\frac{1+\varepsilon}{1-\varepsilon}d_\X(x_k,x_\star) < r 2^{-(k+1)}.
\end{align*}
Thus, all Newton steps remain in $r\B_{x_\star}$ and $x_k \to x_\star$. This implies $\theta_{x_\star}(x_k)\to 0$ for $k\to \infty$ and thus superlinear convergence.
\end{proof}
We now also want to give conditions for local superlinear convergence which can be used for a-priori analysis. 
With the help of a local norm on $\X$ and a fibrewise norm on $\E$ we can use the standard norm of a linear operator $A: E_y \to T_x\X$:
\[
    \Vert A \Vert_{E_{y}\rightarrow T_x\X} := \underset{\|e\|_{E_y}\le 1}{\sup} \Vert Ae\Vert_{T_x\X}.
\]
\begin{proposition}\label{pro:localconvergence}
    Let $F:\X\to\E$ be Newton-differentiable at $x_\star\in\X$ with respect to a Newton derivative $F'\in \Gamma(\mathcal L(T\X,F^*T\E))$, where $F(x_\star)=0$. Assume that for $x\in\X$ the operator norm of $(Q_{F(x)}\circ F^\prime(x))^{-1}$ is uniformly bounded, i.e. there exists $\beta < \infty$ such that
\[ \Vert (Q_{F(x)}\circ F^\prime(x))^{-1} \Vert_{E_{y(x)}\rightarrow T_x\X}  \leq \beta. \]
Then Newton's method converges locally to $x_\star$ at a superlinear rate.
\end{proposition}
\begin{proof}
    Let $\varepsilon \in (0,1)$. Using that the operator norm of $(Q_{F(x)}\circ F^\prime(x))^{-1}$ is bounded and applying Prop. \ref{prop:RetraktionMetric}, we can estimate $\theta_{{x}_\star}({x})$ in a neighborhood $U$ of $x_\star$ as follows:
\begin{align*}
     \theta_{x_\star}(x) &\leq \beta \cdot (1+\varepsilon) \cdot \frac{\|(Q_{F(x)} \circ F^\prime(x)) (R_x^{-1}(x) - R_x^{-1}(x_\star)) - (F(x) - \VT{y}{\leftarrow}(y_\star) F(x_\star))\|_{E_{y(x)}}}{d_\X(x,x_\star)}.
\end{align*}
Since $F(x_\star) = 0$ and $F$ is Newton-differentiable at $x_\star$, we can apply Prop. \ref{QVNewtondiffbar} (i) and obtain
$    \underset{x\to x_\star}{\lim} \theta_{x_\star}(x) = 0$.

\end{proof}
Observe that the quantities, used in Proposition~\ref{pro:localconvergence} depend on the choice of a norm on $E_y$, while the quantities in Proposition~\ref{pro:localconvergenceAI} do not. This suggests that these norms are not strictly necessary to characterize local convergence, one of the core ideas of affine covariant analysis of Newton methods. 
\subsection{Monitoring local convergence}\label{sec:convergenceMonitor}
Let us now define a computable quantity  to monitor local convergence. For this, we set $x= x_k$, $y := y(x) = p(F(x)), \; y_\star := y(x_\star) = p(F(x_\star))$ and consider the affine-covariant quantity $\theta_{x_\star}(x)$ that we have seen in the proof of the local convergence:
\[
    \theta_{{x}_\star}({x}) = \frac{\|(Q_{F(x)} \circ F^\prime(x))^{-1} \left[(Q_{F(x)} \circ F^\prime(x)) (\textbf{x} - \textbf{x}_\star) - (F(x) - \VT{y}{\leftarrow}(y_\star)F(x_\star)) \right]\|_{x}}{\|\textbf{x} - \textbf{x}_\star \|_{x}}.
\]
Since the target point $x_\star$ is computationally not available, we replace $x_\star$ by the next iterate $x_+ := x_{k+1}$, correspondingly the pullback $\textbf{x}_\star$ by $\textbf{x}_{+} = \delta x \in T_{x}{\X}$, and $y_\star$ by $y_{+} = p(F(x_{+}))$. By using $\textbf{x} = 0_{x}$ and the definition of the Newton direction $\delta x=\textbf{x}_+ - \textbf{x} = (Q_{F(x)} \circ F^\prime(x))^{-1}(-F(x))$, we obtain:
\begin{align*}
\theta_{{x}_{+}}({x}) 
& =  \frac{\|(Q_{F(x)} \circ F^\prime(x))^{-1} \left[(Q_{F(x)} \circ F^\prime(x)) (\textbf{x} - \textbf{x}_{+}) - (F(x) - \VT{y}{\leftarrow}(y_{+})F(x_{+})) \right]\|_{x}}{\|\textbf{x} - \textbf{x}_{+} \|_{x}} \\
& \overset{}{=}  \frac{\|(Q_{F(x)} \circ F^\prime (x))^{-1}\VT{y}{\leftarrow}(y_{+})F(x_{+})\|_{x}}{\|\textbf{x} - \textbf{x}_{+} \|_{x}}
\end{align*}
Hence, we can calculate $\theta_{{x}_{+}}({x})$ at the computational cost of the next \textit{simplified Newton direction} $\overline{\delta x_+}$ for our original problem with starting point $x_0 = x_k = x$. Namely, this is the solution $\overline{\delta x_+} \in T_{x}\X$ of the equation
\begin{equation}
\label{vereinfNewton}
    Q_{F(x)} \circ F^\prime (x)\overline{\delta x_+} + \VT{y}{\leftarrow}(y_+)F(x_+) = 0_{{y}}.
\end{equation}
A vector back-transport is needed here since $Q_{F(x)} \circ F^\prime (x)\overline{\delta x_+}$ and $F(x_+)$ do not lie in the same fibre.
Thus, we can rewrite $\theta_{x_+}(x)$ as 
\begin{equation}
    \label{theta}
 \theta_{{x}_+}({x}) = \frac{\|\overline{\delta x_+}\|_{x}}{\| \delta x\|_{x}}.
\end{equation}
The following lemma shows that we can use this to detect local convergence.
\begin{lemma}
    Let $F$ be Newton-differentiable at $x_\star$ with respect to a Newton derivative $F'$, where $F(x_\star)=0$. Assume that the operator norm of $(Q_{F(x)}\circ F^\prime(x))^{-1}$ is uniformly bounded by $\beta < \infty$. Then it holds:
    \[
        \lim_{x\to x_\star}\theta_{x_+}(x) = 0.
    \]
\end{lemma}
\begin{proof}
    Using that the operator norm of $(Q_{F(x)}\circ F^\prime(x))^{-1}$ is bounded and applying Prop. \ref{prop:RetraktionMetric}, we can estimate $\theta_{{x}_+}({x})$ as follows:
    \begin{align*}
        \theta_{x_+}(x) 
        &\leq C \cdot \frac{\|(Q_{F(x)} \circ F^\prime(x)) (R_x^{-1}(x) - R_x^{-1}(x_+)) - (F(x) - \VT{y}{\leftarrow}(y_+) F(x_+))\|_{E_{y}}}{d_\X(x,x_+)}.
   \end{align*}
   Consider the numerator of $\theta_{x_+}(x)$.
   \begin{align*}
    &\|(Q_{F(x)} \circ F^\prime(x)) (R_x^{-1}(x) - R_x^{-1}(x_+)) - (F(x) - \VT{y}{\leftarrow}(y_+) F(x_+))\|_{E_{y}}\\
    &\leq \|(Q_{F(x)} \circ F^\prime(x)) (R_x^{-1}(x) - R_x^{-1}(x_\star)) - (F(x) - \VT{y}{\leftarrow}(y_\star) F(x_\star))\|_{E_{y}}\\
    &\quad + \|(Q_{F(x)} \circ F^\prime(x)) (R_x^{-1}(x_\star) - R_x^{-1}(x_+)) - (\VT{y}{\leftarrow}(y_\star)F(x_\star) - \VT{y}{\leftarrow}(y_+) F(x_+))\|_{E_{y}}.
   \end{align*}
   Since $F$ is Newton-differentiable at $x_\star$ and $F(x_\star)= 0$, we can apply by Prop. \ref{QVNewtondiffbar} (i) and obtain
   \begin{equation}
    \label{eq:ersterSummand}
    \|(Q_{F(x)} \circ F^\prime(x)) (R_x^{-1}(x) - R_x^{-1}(x_\star)) - (F(x) - \VT{y}{\leftarrow}(y_\star) F(x_\star))\|_{E_{y}} = o(d_\X(x,x_\star)).
   \end{equation}
    Since $\VT{y}{\leftarrow}$ and $Q_{F(x)} \circ F^\prime(x)$ are locally bounded and $F$ is locally Lipschitz continuous, we can estimate the second summand by using the triangle inequality, $F(x_\star) = 0$ and Prop. \ref{prop:RetraktionMetric} as 
   \begin{equation}
    \label{eq:zweiterSummand}
    \|(Q_{F(x)} \circ F^\prime(x)) (R_x^{-1}(x_\star) - R_x^{-1}(x_+)) - (\VT{y}{\leftarrow}(y_\star)F(x_\star) - \VT{y}{\leftarrow}(y_+) F(x_+))\|_{E_{y}} \leq \widetilde{C} d_\X(x_\star, x_+).
   \end{equation}
   Since Newton's method converges superlinearly, we obtain
        $d_\X(x_\star, x_+) = o(d_\X(x,x_\star))$
   as $x\to x_\star$. Thus, combining \eqref{eq:ersterSummand} and \eqref{eq:zweiterSummand}, the numerator is of order $o(d_\X(x,x_\star))$ as $x\to x_\star$.
   Estimating the denominator of $\theta_{x_+}(x)$ by 
$d_\X(x,x_+) \geq \vert d_\X(x,x_\star) - d_\X(x_\star,x_+)\vert$,
 we obtain the desired result.
\end{proof}

\section{An Affine Covariant Damping Strategy} 
To obtain convergence from remote initial guesses, Newton's method is typically equipped with a globalization strategy.
Especially in Newton's method applied in optimization some strategies including BFGS, Levenberg-Marquardt and trust region methods were introduced, see e.g. \cite{bertsekas2014constrained,dennis1996numerical,smith2014optimization}. In a linear setting it is well-known that \textit{damped Newton methods} are one way to globalize the convergence of Newton's method (cf. \cite{bertsekas2014constrained, dennis1996numerical,izmailov2014newton,ortega2000iterative}). Here, the step size of the Newton direction is scaled by a factor $\alpha \in (0,1]$. In Newton's method applied in optimization on Riemannian manifolds some strategies including damping and trust region methods were introduced, see e.g. \cite{EfficientDampedNewton, absil2007trust}.

We will also use a damping strategy to globalize our method. The choice of the damping factor is more geometrically motivated here and follows an affine covariant strategy in the spirit of Deuflhard \cite{Deuflhard}. Our way of globalization works for a manifold $\X$, equipped with the Banach-type metrics from Section~\ref{sec:metric}. In our algorithm we only have to evaluate local norms $\|\cdot\|_x$ on tangent spaces, the evalution of $d_\X$ and also of norms $\|\cdot\|_{E_y}$ on $\E$ is not needed. 

To prepare the choice of our damping factor, we first define a so called \textit{Newton path}. Our strategy will then be to follow this path by suitably scaling down the Newton directions until the local convergence area of Newton's method is reached. It will turn out that strict differentiability is an appropriate condition to render this strategy well defined. 
\subsection{Newton paths on Manifolds}
Let us define an algebraic Newton path for $F: \X\to \E$ from a manifold $\X$ to a vector bundle $p:\E\to \Y$, a generalization of the Newton path in linear spaces \cite[Sec. 3.1.4]{Deuflhard}. Denote by $x(0)$ the starting point of the path. For $\alpha \in [0,1]$ and $x(\alpha)\in\X$ we denote by $y(\alpha):=p(F(x(\alpha)))\in \Y$ the base point of $F(x(\alpha))$. Consider a vector back-transport $\VT{y}{\leftarrow} \in \Gamma(\L(\E,E_y))$.
We consider the \textit{Newton path problem}, which is based on the idea of scaling down the residual by a factor of $1-\alpha$:
\begin{equation}
    \label{NewtonPath}
    \NE{\VT{{y(0)}}{\leftarrow}}(F(x(\alpha))) = (1-\alpha) F(x(0)), \; \alpha \in [0,1].
\end{equation}
In contrast to the case of linear spaces, we need a vector back-transport on the left-hand side since $F(x(\alpha))$ and $F(x(0))$ do not lie in the same fibre. The vector back-transport allows us to formulate the Newton path problem in the fixed fibre $E_{y(0)}$, i.e. on a linear space.
We call the mapping $\alpha \to x(\alpha)$ (where it is defined) the \textit{algebraic Newton path} starting at $x(0)$.

\begin{bemerkung}
    For a mapping $G:\mathbb{X}\rightarrow \mathbb{Y}$ between linear spaces $\mathbb{X}$ and $\mathbb{Y}$ the Newton path can be defined alternatively as the solution of the ordinary differential equation \cite[Eq. (3.24)]{Deuflhard}:
\begin{equation*}
G^\prime (x_d(\alpha))x_{d}'(\alpha) + G(x_d(0))=0.
\end{equation*}
We may also generalize this idea and define the \emph{differential Newton path} as the trajectory of the differential equation
\begin{equation}\label{eq:dNewtonPath}
    Q_{F(x_d(\alpha))} \circ F^\prime (x_d(\alpha))x_d^\prime (\alpha) + F(x_d(\alpha)) = 0.
\end{equation}
In contrast to the case of linear spaces the algebraic and the differential Newton path do not coincide in our more general setting.
\end{bemerkung}
\begin{proposition}
    \label{ExistenzNewtonpfad}
    Let $F:\X\to\E$ be strictly differentiable at $x(0)$ and $\NE{\VT{{y(0)}}{\leftarrow}}^\prime(F(x(0)))F^\prime (x(0))\in L(T_{x(0)}\X, E_{y(0)})$ continuously invertible.
    Then for sufficiently small $\alpha$ there exists a solution $x(\alpha)$ of the Newton path problem (\ref{NewtonPath}).
\end{proposition}
\begin{proof}
     Application of the implicit function theorem \cite{zeidler1986nonlinear}, which also holds for strictly differentiable mappings (cf. e.g. \cite[Chap. 25]{schechter1996handbook}), on
     \[
        g(\alpha, x) := \NE{\VT{{y(0)}}{\leftarrow}}(F(x)) - (1-\alpha) F(x(0))
     \] 
    yields the existence of $x(\alpha)$ for sufficiently small $\alpha$.
\end{proof}
By using the chain rule, we obtain the following expression for the computation of  simplified Newton directions $\overline{\delta x^\alpha}\in T_x\X$ at a current iterate $x\in \X$:
\[
    \NE{\VT{{y(0)}}{\leftarrow}}^\prime(F(x(0)))F^\prime (x(0)) \overline{\delta x^\alpha} + \NE{\VT{{y(0)}}{\leftarrow}}(F(x)) - (1-\alpha) F(x(0)) = 0.
\]
Using $\VT{{y(0)}}{\leftarrow}(y(0)) = Id_{E_{y(0)}}$, we calculate the first (simplified) Newton direction $\delta x^\alpha$ at $x=x(0)$ for the Newton path problem as a solution of the following equation:
    \begin{equation}
        \label{eq:firstDirectionPathProblem}
        \NE{\VT{{y(0)}}{\leftarrow}}^\prime(F(x(0)))F^\prime (x(0)) \delta x^\alpha + \underbrace{F(x(0)) - (1-\alpha) F(x(0))}_{=\alpha F(x(0))} = 0.
    \end{equation}
\begin{lemma}
    \label{lem:dampedNewtonStep}
    Let $x\in \X$, $Q \in \Gamma(\L(T\E, p^*\E))$ be a linear connection map, and $\alpha\in (0,1]$. Let $\delta x$ be the current Newton direction, given by solving \eqref{eq:NewtonEquation}.

    Consider the Newton path problem \eqref{NewtonPath} with vector transport $\VT{y}{\leftarrow}\in\Gamma(\L(\E, E_y))$ and assume that $\VT{y}{\leftarrow}$ is consistent with $Q$ at $e=F(x)$. Then the damped Newton direction $\alpha\delta x$ coincides with the first Newton direction $\delta x^\alpha$ of the Newton path problem \eqref{NewtonPath}. Moreover, it is tangent to the algebraic and to the differential Newton path starting at $x(0) = x$. In particular, we obtain
    \begin{equation}
        \label{SchrittNewtonpfad}
        \alpha \delta x=\delta x^\alpha =\alpha x'(0)=\alpha x_d'(0).
    \end{equation}
\end{lemma}
\begin{proof}
    Consider the Newton path problem \eqref{NewtonPath}.
    Implicit differentiation at $x(0)$ with respect to $\alpha$ yields the tangent $x'(0)$ by
    \begin{equation*}
         \NE{\VT{{y(0)}}{\leftarrow}}^\prime(F(x(0)))F^\prime(x(0))x^\prime(0) + F(x(0)) = 0.
    \end{equation*}
    The first (simplified) Newton direction $\delta x^\alpha$ for the Newton path problem with fixed $\alpha$ at the current starting point $x(0) = x$ can be computed by \eqref{eq:firstDirectionPathProblem}:
    \begin{equation*}
        \NE{\VT{{y(0)}}{\leftarrow}}^\prime(F(x(0)))F^\prime (x(0)) \delta x^\alpha + \alpha F(x(0)) = 0.
    \end{equation*}
    Thus, the Newton direction $\delta x^\alpha$ for the Newton path problem is tangent to the algebraic Newton path.
    
    Using the linearity of the connection map $Q$, we can compute the damped Newton direction $\alpha \delta x$ for our original problem at $x(0)$ by
    \[
        Q_{F(x(0))}\circ F'(x(0))(\alpha \delta x) + \alpha F(x(0)) = 0.
    \]
    Comparison with \eqref{eq:dNewtonPath} yields $\alpha\delta x=x_d'(0)$. 
    Using the consistency of $\VT{y}{\leftarrow}$ and $Q$ at $e=F(x(0))$, i.e. $Q_{F(x(0))}=\NE{\VT{{y(0)}}{\leftarrow}}^\prime(F(x(0)))$, this finally yields
    \[
        \delta x^\alpha = \alpha \delta x.
    \]
    In particular, the Newton direction $\delta x$ is tangent to the Newton path at $x(0)$.
\end{proof}
\subsection{Computation of Damping Factors}

Our damping strategy is based on following Newton paths. 
At the current iterate $x$ we consider the algebraic Newton path starting at $x(0)=x$.
Lemma~\ref{lem:dampedNewtonStep} yields the crucial observation that the damped Newton direction is tangential to the Newton path, as along as we choose the vector back-transport and the connection map consistently. 
Thus for small damping factors, the deviation of the damped Newton step from the Newton path is very small. In our case 
the damped Newton step is computed from the Newton direction $\delta x$ and a damping factor $\alpha \in (0,1]$ via a retraction:
\[
    x_+ = R_x(\alpha\delta x).
\]
The idea is now to choose $\alpha$ sufficiently small such that the simplified Newton method for solving \eqref{NewtonPath} with starting point $x(0) = x$ is likely to converge to the target point $x(\alpha)$ on the Newton path, i.e., solves the Newton path problem. By Lemma~\ref{lem:dampedNewtonStep} the damped Newton step coincides with the first Newton step of the Newton path problem. 
Thus we may use the convergence monitor $\theta_{x_+}({x})$ from Section~\ref{sec:convergenceMonitor} at the current iterate $x$ to detect local convergence towards the solution of the Newton path problem. 
According to (\ref{theta}) and using (\ref{SchrittNewtonpfad}), our convergence monitor reads 
\begin{equation}
    \label{eq:theta}
    \theta_{x_+}({x}) =\theta_{R_x(\alpha \delta x)}({x})= \frac{\|\overline{\delta x_+^\alpha}\|_{x}}{\| \alpha \delta x\|_{x}}.
\end{equation}
To evaluate this expression we have to compute the next simplified Newton direction $\overline{\delta x_+^\alpha}$ for the Newton path problem (\ref{NewtonPath}) at the new iterate $x_+$ for the initial guess $x$. This direction is given by the solution of the equation
\begin{equation}\label{eq:secondSimp}
    \NE{\VT{{y(x)}}{\leftarrow}}'(F(x)) F^\prime (x) \overline{\delta x^\alpha_+} + \VT{{y(x)}}{\leftarrow}(y(x_+))F(x_+) - (1-\alpha)F(x) = 0.
\end{equation}
These calculations are repeated, in a back-tracking line search, with decreasing damping factor until $\theta_{x_+}({x}) < \Theta_{acc}$ for some acceptable contraction $\Theta_{acc}$ holds. To achieve this, the damping factor $\alpha$ is readjusted iteratively in the following way. If $\theta_{x_+}({x}) \geq \Theta_{acc}$, choose $\alpha_+ \in (0,1]$ such that for user defined parameters $0 < \Theta_{des} < \Theta_{acc}$ the inequality $\alpha_+ \, \theta_{x_+}({x}) \leq \alpha\,\Theta_{des}$  holds, i.e. we compute the next damping factor by
\begin{center}
$\alpha_+ := \min \left(1, \cfrac{\alpha\Theta_{des}}{\theta_{x_+}(x)}\right)$.
\end{center}
Once a new iterate $x_+=R_x(\alpha\delta x)$ is accepted, we do not iterate towards the corresponding point $x(\alpha)$ on the ``old'' Newton path, but use it as a starting point of a new Newton path, starting at $x_+$. 

We summarise these ideas in Algorithm~\ref{alg:DampedNewton}, which is a modification of \cite[Alg. NLEQ-ERR]{Deuflhard}.

\begin{algorithm}[h] \label{alg:DampedNewton}
    \caption{Affine covariant damped Newton's method}
    \label{GlobalerNewton}
    \begin{algorithmic}[1]
    \REQUIRE $x, \; \alpha \text{ (initial guesses)}; \; \alpha_{fail}, \; TOL$, $\Theta_{des} < \Theta_{acc}$ (parameters); $\VT{y}{\leftarrow}$ (vector back-transport), $Q_{e} = \NE{\VT{y}{\leftarrow}}'(e)$ (consistent connection map), $R_x$ (retraction)
    \REPEAT 
        \STATE solve $\delta x \leftarrow (Q_{F(x)} \circ F^\prime (x))\delta x + F(x) = 0$
            \REPEAT
                \STATE compute $x_+ = R_x(\alpha \delta x)$
                \STATE solve $\overline{\delta x^\alpha_+} \leftarrow Q_{F(x)} \circ F^\prime (x) \overline{\delta x^\alpha_+} + \VT{{y(x)}}{\leftarrow}(y(x_+))F(x_+) - (1-\alpha)F(x) = 0$
                \STATE compute $\theta_{x_+}(x) = \cfrac{\|\overline{\delta x_+^\alpha}\|_x}{\|\alpha \delta x\|_x}$
                \STATE update $\alpha \leftarrow \min \left(1, \cfrac{\alpha\Theta_{des}}{\theta_{x_+}(x)}\right)$
                 \IF{$\alpha < \alpha_{fail}$}
                       \STATE \textbf{terminate:} \grqq Newton's method failed\grqq
                \ENDIF
            \UNTIL{$\theta_{x_+}(x) \leq \Theta_{acc}$}
        \STATE  update $x \leftarrow x_+$
        \IF{$\alpha = 1 \text{ and } \theta_{x_+}(x) \leq \frac{1}{4} \text{ and } \|\delta x\|_x \leq TOL$}
            \STATE \textbf{terminate:} \grqq Desired Accuracy reached\grqq, $x_{out} = x_+$
        \ENDIF
    \UNTIL{maximum number of iterations is reached}
    \end{algorithmic}
    \end{algorithm}

\begin{proposition}
    Let $F$ be strictly differentiable at $x$ and $\VT{y}{\leftarrow}$ be consistent with $Q$. Then there exists $\widehat{\alpha} > 0$ such that
    \[
        \forall \alpha \leq \widehat{\alpha} \; : \; \theta_{R_x(\alpha \delta x)}(x) \leq \Theta_{acc},
    \]
    and the inner loop terminates after finitely many iterations.
\end{proposition}
\begin{proof}
    Without going into the details, the proof of the implicit function theorem, which yields local existence of the Newton path (Prop. \ref{ExistenzNewtonpfad}), shows that the simplified Newton method converges with a linear rate for sufficiently small choice of $\alpha$. By further reducing $\alpha$, the corresponding rate can be chosen arbitrarily small, in particular faster than $\Theta_{acc}$. Thus, since 
    \[
        \theta_{R_{x}(\alpha \delta x)}({x}) \overset{\eqref{eq:theta}}{=} \frac{\|\overline{\delta x_+^\alpha}\|_{x}}{\|\delta x^\alpha\|_{x}},
    \]
    is a lower bound for this rate, we obtain our first statement. 
    
    A sufficiently small choice of $\alpha$ is taken within finitely many steps: if $\theta_{R_x(\alpha\delta x)}(x) > \Theta_{acc}$, we obtain
    \begin{equation*}
        \alpha_+ \leq \cfrac{\alpha \Theta_{des}}{\theta_{R_x(\alpha\delta x)}(x)} \leq \frac{\Theta_{des}}{\Theta_{acc}}\alpha.
    \end{equation*}
    Since $\Theta_{des} < \Theta_{acc}$, this yields $\alpha_+ < C \cdot \alpha \text{ with } C < 1$, so in each step, $\alpha$ is reduced about a fixed factor, until the termination criterion is reached. 
\end{proof}
\begin{bemerkung}
 By the choice of $\Theta_{des} \in (0,1)$ it is possible to adjust how aggressive the step size strategy acts. If $\Theta_{des} \approx 0$, then Newton's method follows the Newton path very closely and takes short steps. If $\Theta_{des}\approx 1$, then larger steps will be taken, but the methods might perform less robustly for highly nonlinear problems. The choices $\Theta_{des}=0.5$ and $\Theta_{acc}=1.1\Theta_{des}$ work well in practice. 
 \end{bemerkung}
 \section{Application to generalized eigenvalue problems}
We will now illustrate our results by an application and a numerical example. For more extensive numerical results and applications we refer to the second part of this paper \cite{weigl2025newton}, where we discuss the application of Newton's method for solving variational problems on manifolds.

Let $A, \, B \in L(X,Y)$ where $(X,\|\cdot\|_X)$ and $(Y,\|\cdot\|_Y)$ are Banach spaces and $\|\cdot\|_X$ is Fr\'echet differentiable on $X\setminus\{0\}$, so that the unit sphere $\mathbb{S}^X := \{ x \in X \mid \|x\|_X = 1\}$ is an embedded submanifold of $X$.  A \emph{generalized eigenvalue problem} consists of finding nonzero vectors $x\in X$ and numbers $\mu \in \mathbb C$ such that 
\[
    Ax = \mu Bx \quad \Leftrightarrow \quad Ax \in \mathrm{span}(Bx). 
\]
Consider the vector bundle $\E$ with base manifold $\Y=Y$, the quotient spaces $E_y = Y / \mathrm{span}(y)$ as fibres, whose elements are denoted by $[v]_y=v+\mathrm{span}(y)\in E_y$, and a vector bundle projection $p: \E \to Y, \; E_y \mapsto y$. Clearly, $0_y=[0]_y=\mathrm{span}(y)$. 
Consider the mapping
\begin{align*}
    F : \mathbb{S}^X &\to \E \\
    x &\mapsto (Bx, [Ax]_{Bx}).
\end{align*}
Then we obtain $y(x) = p(F(x)) = Bx$, and a zero $x\in \mathbb{S}^X$ of $F$ satisfies
\[  
    F(x) = 0_{y(x)} \in E_{y(x)} \Leftrightarrow Ax \in  [0]_{Bx} \subset Y\;  \Leftrightarrow Ax \in \mathrm{span}(Bx).
\]
We will apply Newton's method to $F$ to find a real eigenvalue $\mu \in \R$ (if any exists). Our approach slightly extends the scope of similar methods, which can be found in the literature (cf. e.g. \cite{absil2008optimization,golub2000eigenvalue,sorensen2002numerical}). For simplicity of presentation, we will only describe how to find a single eigenvalue. The presented approach can be straightforwardly carried over to the case of finding multiple eigenvalues by considering vector bundles with fibres $E_W = Y / W$ for an element $W$ of a Grassmann manifold over $Y$.

For the application of Newton's method, a retraction on $\mathbb{S}^X$ is needed to compute the Newton step. For $x\in \mathbb{S}^X$ we use the retraction on the sphere $\mathbb{S}^X$ given by normalization:
\[
    R_x : T_x\mathbb{S}^X \to \mathbb{S}^X, \; \delta x \mapsto \frac{x + \delta x}{\|x+\delta x\|_X}.
\]
Moreover, we need to derive a connection map $Q$ on $\E$ to define the Newton equation
\begin{equation}
    \label{eq:NewtonEquationDual}
    Q_{F(x)}\circ F'(x)\delta x + F(x) = 0_{y(x)},
\end{equation}
which is an equation in the quotient space $E_{y(x)}$.  As we have seen in Lemma \ref{lem:VTconnection} we can derive connection maps by differentiating vector back-transports. We thus have to construct vector transports between quotient spaces
\[
 \VT{y}{\leftarrow}(\eta) : Y / \mathrm{span}(\eta) \to Y/\mathrm{span}(y).
\]
To this end we consider the mapping
\begin{equation}\label{eq:project}
 P(\eta) : Y / \mathrm{span}(\eta) \to Y, \; [v]_\eta \mapsto v-\eta \frac{\eta^*(v)}{\eta^*(\eta)},
\end{equation}
where $v$ is any representative of $[v]_\eta$ and $\eta^*\in Y^* = L(Y, \R)$ has to satisfy $\eta^*(\eta) \neq 0$. A short computation yields that $P(\eta)(v+\alpha \eta)=P(\eta)v$ for all $\alpha \in \R$, so this mapping is well defined, and we also observe $[P(y)v]=[v]$. Then for any choice of $\eta^*$ (which may vary with $\eta$), the mapping
\[
 \VT{y}{\leftarrow}(\eta)[v]_\eta := [P(\eta)v]_y
\]
is a vector back-transport.

Now consider an extension of the fibre component of $F$ given by
\[
    F_Y : \mathbb{S}^X \to Y, \; x\mapsto Ax,
\]
with $F(x)=(y(x),[F_Y(x)]_{y(x)})$ for all $x \in \mathbb{S}^X$.
Differentiating 
$\NE{\VT{y(x)}{\leftarrow}}F({\xi}) = \VT{y(x)}{\leftarrow}(y(\xi))F_Y(\xi) = [P(y(\xi))F_Y({\xi})]_{y(x)}$
with respect to ${\xi}$ at ${\xi} = x$ by using the product rule in the embedding space $Y$ then yields a connection map on $\E$:
\begin{align}
\notag    Q_{F(x)}\circ F'(x)\delta x
    &= \NE{\VT{y(x)}{\leftarrow}}'(F(x))F'(x)\delta x
=  \frac{d}{d\xi}[P(y(\xi))F_Y({\xi})]_{y(x)}\Big|_{\xi=x}\delta x\\
\label{eq:dualConnectionEmbedded}    &= [F_Y'(x)\delta x + P'(y(x))y'(x)\delta x\, F_Y(x)]_{y(x)}
\end{align}
Computing the derivative of $P(\eta)v$ with respect to $\eta$ at $\eta = y$ we obtain by the product rule:
\[
 \left[\frac{d}{d\eta}P(\eta)v|_{\eta = y}\delta y\right]_y =  \left[P'(y)\delta y \,v\right]_y =\left[-\delta y \frac{y^*(v)}{y^*(y)}-y\left(\dots \right)\right]_y = \left[-\delta y \frac{y^*(v)}{y^*(y)}\right]_y.
\]
The term $y(\dots)$ vanishes, taking quotients in $Y/\mathrm{span}(y)$. 

For the implementation of the Newton equation we thus need the derivative of $F_Y : X\to Y$ which is in our example simply given by $F_Y'(x)\delta x = A\delta x$.
Inserting this into \eqref{eq:dualConnectionEmbedded} and using $y'(x)\delta x = B\delta x$, we finally obtain our derivative as follows:
\begin{equation}
\label{eq:covariantDerivativeExample}
Q_{F(x)}\circ F'(x)\delta x = \left[A\delta x-\frac{y^*(Ax)}{y^*(Bx)} B\delta x\right]_{Bx},
\end{equation}
where we still have to choose $y^*$, depending on $x$. An appropriate choice depends on the underlying problem structure:
\begin{itemize}
 \item If $(Y,\langle \cdot,\cdot\rangle_Y)$ is a Hilbert space and $B$ is injective, then the choice $y^*:=\langle Bx,\cdot\rangle_Y$ yields an algorithm for solving generalized eigenvalue problems without any requirements on the symmetry of $B$. 
 We obtain
 \[
  Q_{F(x)}\circ F'(x)\delta x = \left[A\delta x-\frac{\langle Bx,Ax\rangle_Y}{\langle Bx,Bx\rangle_Y} B\delta x\right]_{Bx}.
 \]
 \item In structural mechanics we encounter the case $Y=X^*$, where $B : X \to X^*$ is symmetric positive definite, i.e. $Bv(w)=Bw(v)$ and $Bv(v)>0$ for all $v\neq 0$. Using the canonical embedding $\iota : X\to X^{**}$, given by $\iota x(x^*) := x^*(x)$ we may choose $y^* =\iota x$ and obtain
 \[
  Q_{F(x)}\circ F'(x)\delta x = \left[A\delta x-\frac{Ax(x)}{Bx(x)} B\delta x\right]_{Bx}.
 \]
Frequently, $X$, equipped with the inner product $\langle v,w\rangle_B :=Bv(w)$ is a Hilbert space. Then our algorithm can run on the $B$-unit sphere, implying $Bx(x)=1$ for all iterates. In this case we have
 \[
  Q_{F(x)}\circ F'(x)\delta x = \left[A\delta x-Ax(x)\,B\delta x\right]_{Bx}.
 \]
 
 \item For the standard eigenvalue problem $X=Y=(\R^n,\langle \cdot,\cdot\rangle_2)$ and $B=I$, and thus $y(x)=x$, we may choose $y^*=x^T$, so that \eqref{eq:project} at $\eta = x$ becomes the orthogonal projection $P(x) : X \to x^\perp=T_x\mathbb S^X$ and 
 \[
  Q_{F(x)}\circ F'(x)\delta x = \left[A\delta x-x^TAx\,\delta x\right]_x.
 \]
Then our approach is equivalent to the algorithm, presented in \cite[6.4.4]{absil2008optimization}, where a zero of the vector field $\xi(x) := P(x)Ax$ is computed by Newton's method. 
\end{itemize}

\paragraph{Implementation.} For the computation of the Newton direction $\delta x$ we have to solve the Newton equation \eqref{eq:NewtonEquationDual} while taking into account that $\delta x \in T_x\mathbb{S}^X\subset X$, which we assume to be a Hilbert space with inner product $\langle\cdot,\cdot\rangle_X$.  In the following we consider a formulation of the Newton equation as an equation in $Y$. 
If $Bx\neq 0$, there exists a unique $\lambda\in \R$ such that 
\begin{align*}
         F_Y'(x)\delta x + P'(y(x))\delta x\,F_Y(x) - \lambda Bx + F_Y(x) &= 0 \text{ in } Y \quad \text{and} \quad  \langle x, \delta x\rangle_X = 0\\
          \Leftrightarrow \quad [F_Y'(x)\delta x + P'(y(x))\delta x\,F_Y(x) + F_Y(x)]_{y(x)} &= [0]_{y(x)} \subset Y \quad \text{and} \quad  \langle x, \delta x\rangle_X = 0\\
        \Leftrightarrow \quad  Q_{F(x)} \circ F'(x)\delta x + F(x) &= 0_{y(x)} \text{ in } E_{y(x)}\quad \text{and} \quad  \delta x \in T_x\mathbb{S}^X
\end{align*}
For our numerical example we consider $X=Y=\R^n$ equipped with the euclidean inner product $\langle \cdot, \cdot \rangle_2$. In our example we use the matrix representation of \eqref{eq:covariantDerivativeExample} and the representations of the tangent spaces and fibres. Then, we can write the Newton equation as a saddle point system:
\begin{equation} 
    \label{eq:NewtonEquationGenEig}
            \begin{pmatrix}
                A-\frac{\langle Bx, Ax\rangle_2}{\langle Bx,Bx\rangle_2}B & -Bx\\
                x^T & 0
            \end{pmatrix}
            \begin{pmatrix}
                \delta x \\ \lambda
            \end{pmatrix}
            + \begin{pmatrix}
                Ax \\ 0
            \end{pmatrix}
            = \begin{pmatrix}
                0 \\ 0
            \end{pmatrix}.
\end{equation}
To increase the numerical stability we use an estimate $\widehat\lambda$ for the Lagrangian multiplier $\lambda$ given by
\[
    \widehat\lambda(x) := \frac{\langle Ax, Bx\rangle_2}{\langle Bx,Bx\rangle_2}
\]
for $x \in \mathbb{S}^X$
 and consider instead the linear system
\begin{equation}
    \label{eq:NewtonEquationGenEigStable}
            \begin{pmatrix}
                A-\widehat\lambda(x)B & -Bx\\
                x^T & 0
            \end{pmatrix}
            \begin{pmatrix}
                \delta x \\ \delta\lambda
            \end{pmatrix}
            + \begin{pmatrix}
                Ax - \widehat\lambda(x) Bx \\ 0
            \end{pmatrix}
            = \begin{pmatrix}
                0 \\ 0
            \end{pmatrix}.
\end{equation}
Then, in the absence of round-off error, \eqref{eq:NewtonEquationGenEig} and \eqref{eq:NewtonEquationGenEigStable} provide the same solution for the Newton direction $\delta x$. The direction $\delta \lambda \in \R$ is just an auxiliary variable, not used in the iteration. As $Ax - \widehat\lambda(x)Bx$ tends to zero during the iteration, the estimate $\widehat\lambda(x)$ also gives an approximation of a generalized eigenvalue.

Algorithm~\ref{GlobalerNewton} is implemented in the programming language \texttt{Julia}~\cite{BezansonEdelanKarpinskiViral:2017} using the software library \texttt{Manifolds.jl}~\cite{Manifolds}, which provides abstractions for nonlinear manifolds. This work was performed in cooperation with Ronny Bergmann, see also our companion paper~\cite{weigl2025newton}, and it is planned to make the code available in the framework of \texttt{Manopt.jl}\cite{Bergmann2022}.

As an illustrating example we choose $n=101$ and $A,B \in \R^{n\times n}$ with
\[
        A_{ij} := \begin{cases}
            i & \text{if } i=j \\
            1 & \text{else}
        \end{cases}, \qquad B_{ij} := \begin{cases}
            -1 & \text{if } i > j \\
            \phantom{-}0 & \text{else}
        \end{cases}.
\]
Then, we solve the generalized eigenvalue problem $Ax = \mu Bx$ by applying Newton's method with damping (cf. Alg. \ref{alg:DampedNewton} with $\Theta_{des} = 0.5$, $\Theta_{acc} = 1.1\Theta_{des}$, $TOL = 10^{-12}$) to the corresponding mapping $F$. As initial guesses we choose the first unit vector $x_0 = \begin{pmatrix}
    1, \, 0, \, \dots, \, 0 
\end{pmatrix}^T$.
Figure~\ref{fig:ConvergenceAndStepsizes} shows the local superlinear convergence that we expected to see due to Prop. \ref{pro:localconvergenceAI}, and the stepsizes chosen by the affine covariant damping strategy. 
\begin{figure}[tbp]
    \begin{minipage}[t]{0.49\textwidth}
        \begin{tikzpicture}
        \begin{axis}[
            xlabel={Iteration},
            ylabel={$\Vert \delta x\Vert$},
            ymode=log,
            grid=both,
            width=12cm,
            height=8cm,
            scale=0.55,
            xtick={1,2,3,4,5,6,7,8,9,10,11,12,13},
            ytick distance=10^4
        ]
                \addplot table [
            col sep=comma,
            header=false,
            x expr=\coordindex+1,  
            y index=0            
            ]{data/norm_newton_direction_eigenvalues.csv};
        \end{axis}
        \end{tikzpicture}
    \end{minipage}
    \begin{minipage}[t]{0.49\textwidth}
                \begin{tikzpicture}
        \begin{axis}[
            xlabel={Iteration},
            ylabel={$\alpha$},
            grid=both,
            width=12cm,
            height=7.5cm,
            scale=0.6,
            xtick={1,2,3,4,5,6,7,8,9,10,11,12,13},
            ytick distance=0.1
        ]
                \addplot table [
            col sep=comma,
            header=false,
            x expr=\coordindex+1,
            y index=0             
            ]{data/stepsize_eigenvalues.csv};
        \end{axis}
        \end{tikzpicture}
        \label{fig:stepsizesRod}
    \end{minipage}
    \caption{Superlinear convergence of Newton's method (left) and damping factors $\alpha$ chosen by affine covariant damping strategy (right).}
    \label{fig:ConvergenceAndStepsizes}
\end{figure}

\section{Conclusion and future research}

In our study of Newton's method for a mapping $F:\X \to \E$ from a manifold into a vector bundle we have found a number of structural insights. The most basic distinction from the classical case is the need for retractions $R_x$ on $\X$ and connections $Q_e$ on $\E$ to render the Newton steps well defined. Together with a geometric version of Newton-differentiability this already allows a local convergence theory of Newton's method. A Banach type version of a Riemannian metric can be used as a flexible framework to formulate superlinear convergence. 

For the development of further algorithmic strategies, like monitoring local convergence or globalization, vector back-transports $\VT{y}{\leftarrow}$ on $\E$ are required. They make it possible to compare residuals from different fibres and to compute simplified Newton steps. We propose an affine covariant globalization scheme, which works purely with quantities that can be computed in terms of the domain $\X$. In the global regime, where residuals are not small, it is necessary that the employed connection is consistent with the vector back-transport in order to guarantee that the Newton direction is tangential to the algebraic Newton path. 
Our general approach can be used to tackle various classes of problems numerically, posed on manifolds of finite and infinite dimension and thus opens the door for future research. Possible applications are the simulation and optimal control of geometric variational problems and differential equations, problems of stationary action, or shape optimization. Depending on the structure of the problem, alternative globalization schemes can be devised, for example, residual based schemes or descent methods from nonlinear optimization.

\bibliographystyle{plain}
\bibliography{references}
\end{document}